\documentclass[aos,preprint]{imsart} 

\RequirePackage[OT1]{fontenc}
\RequirePackage{amsthm,amssymb,amsfonts,amsmath}
\RequirePackage[numbers]{natbib}
\RequirePackage[colorlinks,citecolor=blue,urlcolor=blue,pdffitwindow=false,  pdfstartview={FitH}]{hyperref}
\usepackage{epsfig}
\usepackage{stmaryrd}
\usepackage{graphics}
\usepackage{subfigure}
\usepackage[T1]{fontenc}
\usepackage{graphicx}
\usepackage{epstopdf}
\usepackage{hyperref}
\usepackage{pdfsync}
\usepackage{color}
\usepackage{dsfont}
\newcommand{\N}{\ensuremath{\mathbb N} }
\newcommand{\R}{\ensuremath{\mathbb R} }
\newcommand{\C}{\ensuremath{\mathbb C} }
\newcommand{\Z}{\ensuremath{\mathbb Z} }
\newcommand{\PP}{\ensuremath{{\mathbb P}} }
\newcommand{\EE}{\ensuremath{{\mathbb E}} }

\newcommand{\cF}{{\cal F}}
\newcommand{\cG}{{\cal G}}
\newcommand{\cH}{{\cal H}}

\newcommand{\cL}{{\cal L}}
\newcommand{\cM}{{\cal M}}
\newcommand{\cN}{{\cal N}}
\newcommand{\cO}{{\cal O}}
\newcommand{\cP}{{\cal P}}

\newtheorem{theo}{Theorem}
\newtheorem{proposition}{Proposition}
\newtheorem{lemma}{Lemma}

\newtheorem{remark}{Remark}

\numberwithin{equation}{section}
\numberwithin{ass}{section}
\numberwithin{theo}{section} \numberwithin{proposition}{section}
\numberwithin{lemma}{section}
\numberwithin{remark}{section}

\newcommand{\re}{\Re\mathrm{e}} 
\def\M{\mathfrak{M}([0,1])}
\def\Mnu{\mathfrak{M}_{\nu}([0,1])(A)}
\def \i{\mathfrak{i}}
\def \xikn{\xi_{n}} 

\newcommand{\1}{\ensuremath{\textbf{1}}}

\newcommand{\underbraceabs}[2]{\left|\vphantom{#1}\right. \underbrace{#1}_{#2} \left.\vphantom{#1}\right| }

\arxiv{arXiv:0000.0000}

\startlocaldefs
\numberwithin{equation}{section}
\theoremstyle{plain}

\endlocaldefs

\begin{document}

\begin{frontmatter}
\title{Bayesian methods in the Shape Invariant Model (II): Identifiability and posterior contraction rates on functional spaces}
\runtitle{Bayesian methods in the Shape Invariant Model (II)}

\begin{aug}
\author{\fnms{Dominique} \snm{Bontemps}\thanksref{t2}\ead[label=e1]{dominique.bontemps@math.univ-toulouse.fr}} \and
\author{\fnms{S\'ebastien} \snm{Gadat}\thanksref{t2}\ead[label=e2]{sebastien.gadat@math.univ-toulouse.fr}}

\thankstext{t2}{The authors acknowledge the support of the French Agence Nationale de la Recherche (ANR) under references ANR-JCJC-SIMI1 DEMOS and ANR Bandhits.}
\runauthor{D. Bontemps and S. Gadat}

\affiliation{Institut Math\'ematiques de Toulouse, Universit\'e Paul Sabatier}

\address{Institut Math\'ematiques de Toulouse, 
 Universit\'e Paul Sabatier\\118 route de Narbonne
F-31062 Toulouse Cedex 9 FRANCE\\
\printead{e1}\\\printead{e2}\\}

\end{aug}

\begin{abstract}
In this paper, we consider the so-called Shape Invariant Model which stands for the estimation of a function $f^0$ submitted to a random translation of law $g^0$ in a white noise model. We are interested in such a model when the law of the deformations is {\em unknown}. 
We aim to recover the law of the process $\PP_{f^0,g^0}$ as well as $f^0$ and $g^0$. 


We first provide some identifiability result on this model and then adopt a Bayesian point of view. In this view, we find some prior on $f$ and $g$ such that the posterior distribution  concentrates around the functions $f^0$ and $g^0$ when $n$ goes to $+\infty$, we then obtain a contraction rate of order a power of $\log(n)^{-1}$.
We also obtain a  lower bound on the model for the estimation of $f^0$ and $g^0$ in a frequentist paradigm which also decreases following a power of $\log(n)^{-1}$.
\end{abstract}



\end{frontmatter}

\section{Introduction}

We are interested in this work in the so-called Shape Invariant Model (SIM). Such model aims to describe a statistical process which involves a deformation of a functional shape according to some randomized geometric variability. 
Such a model possesses various applications in biology, genetic, imaging science, econometry (one should refer to \cite{BG_1} for a more detailed list of possible applications and references).

In the mathematical community, it has also received a large interest as pointed by the numerous references on this subject (see also \cite{BG_1}), and various methods have been developed: $M$-estimation, multi-resolution and harmonic analysis, geometry or semi-parametric statistics. In our study, we consider the general case of an unknown shape submitted to a randomized deformation whose law is also unknown.
 We adopt here a Bayesian point of view and want to extend the results obtained on the probability laws in \cite{BG_1} to the functional elements which parametrize the Shape Invariant Model. Hence, starting from the strategy used in \cite{BG_1}, we aim to recover a contraction rate of the posterior distribution  on the functional objects themselves (shape and mixture law of the deformations), when the number of observations $n$ is growing to 
 $+ \infty$.
We will use in the sequel quite standard Bayesian non parametric methods already introduced in \cite{BG_1} to obtain the frequentist consistency and some contraction rates of the Bayesian procedures.
We will be interested in this paper on the consistency around the functional objects $f^0$ and $g^0$ of the posterior distribution where $f^0$ is the unknown shape to recover, and $g^0$ is the distribution of the nuisance parameter which deforms the shape. In this view, it will be necessary to consider smooth classes for both the shape $f$ and the mixture $g$. This last point is quite different from the situation studied in \cite{BG_1} where any mixture distributions (not necessarily smooth) were considered. Thus, we are naturally driven to consider prior on smooth densities: Dirichlet priors used in \cite{BG_1} will then become useless although Gaussian process priors considered in \cite{vdWvZ} will be of first importance. Our approach will be adaptive on $f$ but not on $g$: we will assume in the paper the smoothness parameter (denoted $s$) of $f$ {\em unknown} but the smoothness parameter $\nu$ of $g$ will be assumed {\em known}.

The paper is organised as follows. Section \ref{sec:model_not_result} recalls a reduced description of the Shape Invariant Model, provides some notations for mixture models, describes our prior on $(f,g)$ and  gives our main results. 
Section \ref{sec:prior} briefly describes the behaviour of the posterior distribution for the new prior defined on $(f,g)$ and main arguments relies on the previous work \cite{BG_1}.
Section \ref{sec:identifiability} provides 
 some general identifiability results and up to these identifiability conditions, shows the posterior contraction on the functional objects themselves. At last, section \ref{sec:lowerbound}  exploits the Fano Lemma and establishes a lower bound result of reconstruction in a frequentist paradigm. We end the paper with a short concluding section.


\section{Model, notations and main results} \label{sec:model_not_result}

\subsection{Statistical settings}

\paragraph{Shape Invariant Model}
We briefly summarize the notations introduced in \cite{BG_1} for the random Shape Invariant Model (shortened as SIM in the sequel). We assume $f^0$ to be a "mean pattern" which belongs to a subset $\cF$ of smooth functions. 
We also consider a probability measure $g^0$ which generates random shifts denoted $(\tau_j)_{j=1 \ldots n}$.
We observe $n$ realizations of  noisy and randomly shifted complex valued curves $Y_1, \ldots, Y_n$ coming from the following white noise model
\begin{equation}\label{eq:model}
\forall x \in [0,1] \quad \forall j=1 \ldots n \qquad dY_j(x): = f^0(x-\tau_j) dx + \sigma dW_j(x).
\end{equation}
Here, $(W_j)_{j = 1 \ldots n}$ are independent complex standard Brownian motions on $[0,1]$, 
 the noise level $\sigma$ is kept fixed in our study and is set to $1$.
 
In the sequel, $f^{-\tau}$ 
is the function $x\mapsto f(x-\tau)$. 
Complex valued curves are considered here for the simplicity of notations.
We  intensively use the notation  "$\lesssim$" which refers to an inequality up to a multiplicative absolute constant. In the meantime, $a\sim b$  stands for $a/b \longrightarrow 1$.

\paragraph{Functional setting and Fourier analysis}
Without loss of generality, the function $f^0$ is assumed to be periodic with period $1$ and to belong to a subset $\cF$ of $L^2_\C([0,1])$,
endowed with the norm $\|h\|:=\int_0^1 |h(s)|^2 ds$. 
The complex Fourier coefficients of $h$ are denoted $\theta_{\ell}(h)$, $\ell\in\Z$. 
We will often use the parametrisation in $\cF$ 
through the Fourier expansion and  will simply use the notation $(\theta_{\ell})_{\ell \in \Z}$ instead of $(\theta_{\ell}(h))_{\ell \in \Z}$. 
Since we aim to consider smooth elements $f$, we are interested by some Sobolev spaces and introduce the following useful set of functions (which is a subspace of a Sobolev space):
$$
\cF_s
:=\left\{f\in L^2_\C([0,1]) \quad \vert \quad \theta_1(f)>0 \quad  \text{and} \quad \sum_{\ell \in \Z} (1+|\ell|^{2s}) |\theta_\ell(f)|^2 
< + \infty  \right\}.
$$
A second useful set is defined as truncated elements of the former set:
$$
\cF^{\ell}
:=\left\{f\in L^2_\C([0,1]) \quad \vert \quad \theta_1(f)>0 \quad  \text{and} \quad \forall |k| > \ell \qquad \theta_k(f) = 0\right\}.
$$

We will explain in section \ref{sec:identifiability} why such a restriction is natural for the identifiability of the SIM. In the sequel, we denote the Sobolev norm
$$
\|\theta\|_{\cH_1} := \sqrt{ \sum_{\ell\in \Z} \ell^2 |\theta_\ell|^2}.
$$

 In contrary to the picture described in \cite{BG_1}, we consider only \textit{smooth} densities, characterised by a regularity parameter $\nu$ and a radius $A$:
$$
\Mnu := \left\{ g \in \M \quad \vert \quad \sum_{k \in \Z} k^{2 \nu} |\theta_k(g)|^2  <A^2\right\},
$$
where  $\M$ is the set of probability on $[0,1]$ and $\|.\|$ is the $L^2_{\R}$ norm. At last, we will also need the set
$$
\M^{\star} := \left\{ g \in \M  \quad \vert \quad \forall k \in \mathbb{Z} \quad  \theta_k(g) \neq 0\right\}.
$$

\paragraph{Bayesian framework}

We consider functional objects $(f^0,g^0)$ belonging to $\cF_s \otimes \Mnu$ and for any couple $(f,g) \in \cF_s \otimes \Mnu$,  equation \eqref{eq:model} describes the law of one continuous curve, whose law is denoted $\PP_{f,g}$. 
We denote $\cP$ the set of probability measures over the sample space, described by \eqref{eq:model} when $(f,g)$ varies into $\cF_s \otimes \Mnu$.
 Given some prior distribution $\Pi_n$ on $\cP$, 
we are interested in the asymptotic behaviour of the posterior distribution defined by
 $$
 \Pi_n\left(B|Y_1,\ldots, Y_n\right)  = \frac{\int_{B} \prod_{j=1}^n p(Y_j) d\Pi_n(p)}{\int_{\cP} \prod_{j=1}^n p(Y_j) d\Pi_n(p)}.
 $$

\paragraph{Mixture  model}
According to equation \eqref{eq:model}, we can write in the Fourier domain that
\begin{equation}\label{eq:model_fourier}
\forall \ell \in \Z \quad \forall j \in \{1 \ldots n\}  \qquad 
\theta_{\ell}(Y_j) = \theta^0_{\ell} e^{- \i 2 \pi j \tau_j} + \xi_{\ell,j},
\end{equation}
where $\theta^0: =(\theta^0_{\ell})_{\ell \in \Z}$ denotes the true unknown Fourier coefficients of $f^0$. 
The variables $(\xi_{\ell,j})_{\ell,j}$ are independent standard (complex) Gaussian random variables: $\xi_{\ell,j} \sim_{i.i.d.} \cN_{\C}(0,1), \forall \ell,j$. 
For any $p$ dimensional complex vector $z$,  $\gamma$ will refer to $\gamma(z) := \pi^{-p} e^{-\|z\|^2}$, the density of the standard complex Gaussian centered distribution $\cN_{\C^p}(0, Id)$, and $\gamma_{\mu}(.) := \gamma(.-\mu)$ is the density of the standard complex Gaussian with mean $\mu$.

For any frequence $\ell$,
$\theta_\ell(Y)$ follows a mixture of complex Gaussian standard variables 
 $
\theta_\ell(Y) \sim \int_{0}^1 \gamma_{\theta^0_{\ell} e^{- \i 2 \pi \ell \varphi}} dg(\varphi).
$
Thus, in the sequel for any phase $\varphi \in [0,1]$ we define the useful the notation
$$
\forall \theta \in \ell^2(\Z)\quad
\forall \ell \in \Z \qquad (\theta \bullet \varphi)_{\ell}: = \theta_{\ell} e^{- \i 2 \pi \ell \varphi}.
$$
Thus, the law of the infinite series of Fourier coefficients of $Y$ is
\begin{equation}\label{eq:not_mixture}
\theta(Y) \sim \int_{0}^1 \gamma_{\theta^0 \bullet \varphi}(.) dg(\varphi).
\end{equation}
For any  $\theta \in \ell^2_\C(\Z)$ and for any $g \in \Mnu$, $\PP_{\theta,g}$ will refer to the law of the vector of $\ell^2(\Z)$ described by 
the location mixture of Gaussian variables. Following a notation shortcut,  $\PP_{f,g}$ is the  law on curves derived from $\PP_{\theta,g}$.

\subsection{Bayesian prior in the randomly shifted curves model}\label{section:prior}

We detail  here the Bayesian prior $\Pi_n$ on $\cP$ used to obtain a suitable concentration rate. 
The two parameters $f$ and $g$ are picked independently at random following the next prior distributions. The shape $f$ is sampled according to $\pi$ and the deformation law $g$ is sampled according to $q_{\nu,A}$, both defined below. The prior distribution $\pi$ will be adaptive w.r.t. the Sobolev space where $f^0$ is living. The prior distribution $q_{\nu,A}$ will be dependent on some knowledge on $g^0$: its regularity and an upper bound of its norm.

\paragraph{Adaptive prior on $f$} 
The prior is mainly described in \cite{BG_1}  and defined on $\cF_s$ through 
$$
\pi := \sum_{\ell \geq 1} \lambda(\ell) \pi_\ell.
$$
Given any integer $\ell$, the idea is to decide to randomly switch on with probability $\lambda(\ell)$ all the Fourier frequencies from $-\ell$ to $+ \ell$. We denote $\left| \cN_\R\right|(0,\sigma^2)$ the law of the absolute value of a real centered Gaussian variable of variance $\sigma^2$.
Then,  $\pi_{\ell}$ is  defined by $\pi_{\ell}:=\otimes_{k \in \Z} \pi_{\ell}^k$ and
$$
\forall k \in \Z \qquad \pi^k_{\ell}= 1_{|k| > \ell} \delta_{0}+ 1_{k \neq 1} \times 1_{|k| \leq \ell} \cN_\C(0,\xikn^2) + 1_{k=1} \left|\cN_\R\right|(0,\xikn^2). 
$$
The law on the first Fourier coefficient is slightly different from the one used in \cite{BG_1} in order to belong to the identifiability class obtained through $\cF_s$ (we have to impose the strict positivity of $\theta_1$, the first Fourier coefficient).
The randomisation of the selected frequencies is done using $\lambda$, a probability distribution on $\N^{\star}$ which satisfies, for some $\rho \in (1,2)$:
 $$\exists (c_1,c_2) \in \R_+ \quad \forall \ell \in \N^{\star}  \qquad 
e^{-c_1 \ell^2\log^\rho \ell} \lesssim \lambda(\ell) \lesssim e^{-c_2 \ell^2 \log^\rho \ell}.
$$
In the sequel, we use a special case of the prior proposed in \cite{BG_1}:
\begin{equation}\label{eq:variance_prior}
\xi_n^2 := n^{-1/4} (\log n)^{-3/2}.
\end{equation}

\paragraph{Non-adaptive prior on $g$} The prior on $\Mnu$ will be the main difference with the one given in \cite{BG_1}. We  propose to use another prior in this work since we will need some smoothness result on $g$ to push our result further than a simple contraction on laws. Such smoothness is not compatible with Dirichlet priors and even kernel convolution with Dirichlet process seems problematic in our situation. 
Thus, we have chosen to use some prior based on gaussian process. \textit{More precisely, we assume in all the paper that we know the smoothness parameter $\nu$ of $g^0$, as well as the radius $A$ of the Sobolev balls where $g^0$ is living.}

Given $\nu \geq 1/2$ and $A>0$, we define the integer $k_{\nu} := \lfloor \nu -1/2 \rfloor $ to be the largest integer smaller  than $\nu-1/2$. We follow the strategy of section 4 in \cite{vdWvZ} and the important point is that we have to take into account the $1$-periodicity of the density $g$, as well as its regularity.
In this view, we denote $B$ a Brownian bridge between $0$ and $1$. The Brownian bridge can be obtained from a Brownian motion trajectory $W$ using
$B_t=W_t - t W_1$. Then, 
For any continuous function $f$ on $[0,1]$, we define the linear map

$$
J(f): t \longmapsto \int_{0}^t f(s) ds - t \int_{0}^1 f(u) du,
$$
and all its composition are $J_k = J_{k-1} \circ J$.
Moreover, in order to adapt our prior to the several derivatives of $g$ at points $0$ and $1$, we use the family of maps $(\psi_j)_{j=1 \ldots k_{\nu}}$ defined as
$$
\forall t \in [0,1] \qquad 
\psi_k(t) := \sin(2\pi k t) + \cos (2\pi k t).
$$

Our prior is now built as follows, we first sample a real Brownian bridge $(B_\tau)_{\tau \in [0,1]}$ and $Z_1, \ldots Z_{k_\nu}$ independent real standard normal random variables. This enables to generate the Gaussian process
\begin{equation}\label{eq:prior_process}
\forall \tau \in [0,1] \qquad w_\tau := J_{k_{\nu}}(B)(\tau) + \sum_{i=1}^{k_{\nu}}  Z_i \psi_i(\tau).
\end{equation}
 Given  $(w_{\tau},\tau \in [0;1])$  generated by \eqref{eq:prior_process}, we build $p_w$ through
\begin{equation}\label{eq:log_gaussian}
\forall \tau \in [0;1] \qquad p_{w}(\tau) := \frac{ e^{w_\tau}}{\int_{0}^1 e^{w_\tau} d\tau}.
\end{equation}
Hence, a prior on Gaussian process yields a prior on densities on $[0;1]$ and $p_w$ inherits of the smoothness $k_{\nu}$ of the Gaussian process $\tau \mapsto w_{\tau}$. According to our construction, we now consider the restriction of the prior defined above to the Sobolev balls of radius $2A$. This finally defines a prior distribution $q_{\nu,A}$ on $\mathfrak{M}_{\nu}([0,1])(2A)$.

\subsection{Main results}
Using the prior distribution $\Pi_n := \pi \otimes q_{\nu,A}$, 
we will first establish the following result on the randomly SIM.
\begin{theo}\label{theo:posterior_shift}
Assume that 
$f^0 \in \cH_s$ with $s \geq 1$ and $g^0 \in \Mnu$, then there exists a sufficiently large $M$ such that
$$
\Pi_n \left\{ \PP_{f,g} \text{ s.t. } d_H(\PP_{f,g} ,\PP_{f^0,g^0}) \leq M \epsilon_n \vert Y_1, \ldots Y_n \right\} = 1 + \cO_{\PP_{f^0,g^0}}(1)
$$
when $n \longrightarrow + \infty$, where for an explicit $\kappa>0$: 
$$
\epsilon_n = n^{- \left[ \frac{\nu}{2\nu+1} \wedge \frac{s}{2s+2} \wedge \frac{3}{8}\right] } \log (n) ^{\kappa}.
$$
\end{theo}

We derive also in this paper a second results on the objects $f \in \cH_s$ and $g \in \Mnu$ themselves. The first one concerns the identifiability of the model and is stated below. 

\begin{theo}\label{theo:ident}
The Shape Invariant Model is identifiable as soon as $(f^0,g^0) \in \cF_s\times \M^{\star}$: the canonical application 
$$\mathcal{I}\quad : \quad  (f^0,g^0) \in \cF_s \times \M^{\star} \longmapsto \PP_{f^0,g^0} \qquad \text{is injective}.$$
\end{theo}

According to this identifiability result, we can derive a somewhat quite weak result on the posterior convergence towards the true objects $f^0$ and $g^0$.

\begin{theo} \label{theo:semiparametric_rates}
The two following results hold.
i)
Assume that 
$f^0 \in \cF_s$ with $s \geq 1$ and $g^0 \in \Mnu$ with $\nu>1$, then there exists a sufficiently large $M$ such that
$$
\Pi_n \left\{ g \quad s.t. \, \|g-g^0\| \leq M \mu_n \vert Y_1, \ldots Y_n \right\} = 1 + \cO_{\PP_{f^0,g^0}}(1)
$$
with the contraction rate
$
\mu_n = \left( \log n \right)^{-\nu}.
$

ii)
In the meantime, assume that $g^0 \in \Mnu$ satisfies the inverse problem assumption:
$$
\exists (c) > 0 \quad \exists \beta > \nu+\tfrac{1}{2} \quad \forall k \in \mathbb{Z} \qquad |\theta_k(g^0)| \geq c k^{-\beta}
$$
then we also have 
$$
\Pi_n \left\{ f \quad s.t. \, \|f-f^0\| \leq M \tilde{\mu}_n \vert Y_1, \ldots Y_n \right\} = 1 + \cO_{\PP_{f^0,g^0}}(1)
$$
when $n \longrightarrow + \infty$. Moreover, the contraction rate $\tilde{\mu}_n$ is given by
$$
\tilde{\mu}_n = \left( \log n \right)^{- \frac{4s\nu}{2s+2\beta+1}}.
$$
\end{theo}

We do not know if such kind of result may be asymptotically optimal since a frequentist minimax rate does not seem identified for the randomly shifted curve model when both $f$ and $g$ are unknown. 
In Section \ref{sec:lowerbound}, we will stress the fact that it is indeed impossible to obtain frequentist convergence rates better than some power of $\log n$, even if our lower bound does not exactly match  with the upper bound obtained in the previous result.

\begin{theo} 

Assume that $(f^0,g^0) \in \cF_s \times \Mnu$, then there exists a sufficiently small $c$ such that the minimax rate of estimation over $ \cF_s \times \Mnu$ satisfies
$$
\liminf_{n \longrightarrow + \infty} \left( \log n\right)^{2s+2}
\inf_{\hat{f} \in \cF_s}  \sup_{(f,g) \in \cF_s \times \Mnu} \|\hat{f} - f\|^2 \geq c,
$$
and
$$
\liminf_{n \longrightarrow + \infty} \left( \log n\right)^{2\nu + 1}
\inf_{\hat{g} \in \cF_s}  \sup_{(f,g) \in \cF_s \times \Mnu} \|\hat{g} - g\|^2 \geq c .
$$
\end{theo}

\noindent
This result is far from being contradictory with the polynomial rate obtained in Theorem \ref{theo:posterior_shift}. One can make  at least three remarks:
\begin{itemize}
\item The first result provides a contraction rate on the probability distribution in $\mathcal{P}$ and not on the functional space $\cF_s$.
\item The link between $(f^0,g^0)$ and $\PP_{f^0,g^0}$ relies on the identifiability of the model, and the lower bound is derived from a net of functions $(f_i,g_i)_{i}$, which are really hard to identify according to the application $\mathcal{I}: (f,g) \mapsto \mathbb{P}_{f,g}$. On this net of functions, the injection is very ``flat'' and the two by two differences of $\mathcal{I}(f^i,g^i)$  are as small as possible and thus the pairs of functions $(f^i,g^i)$ become very hard to distinguish.
\item In fact, \cite{BG10} have shown that in the SIM, when $n \longrightarrow + \infty$, it is impossible to recover the unknown true shifts. The abrupt degradation between the polynomial rates on probability laws in $\cP$  and the logarithmic rates on functional objects in $\cF_s \times \Mnu$ also occurs owing to such a reason. 
One may argue that such an artefact could be avoided if one chooses a different distance on $\cF_s$, which may be better suited to our framework, such as
$$
d_{Frechet}(f_1,f_2) := \inf_{\tau \in [0,1]} \left\|f_1^{-\tau}-f_2\right\|.
$$ 
We do not have purchased further our investigations with this distance on $\cF_s$ but it would certainly be  a nice progress to obtain the posterior contraction using such a distance. We expect a polynomial rate, 
but it is clearly an open (and probably hard) task.
\end{itemize}

\section{Contraction Rate of the posterior distribution} \label{sec:prior}

We provide in this section a short proof of  Theorem \ref{theo:posterior_shift} since it is almost an extension of the result obtained in \cite{BG_1} for a more general class of mixture models.
We first recall a useful result established in \cite{BG_1} which links the total variation distance between $\PP_{f,g}$ and $\PP_{\tilde{f},g}$ and the norm of $f-\tilde{f}$ in $L^2$.
\begin{lemma} \label{lemma:dVT_sur_f}
 Let $f$ and $\tilde{f}$ be any functions in $L^2_\C([0, 1])$, $g$ be any shift distribution in $\M$, then
 \begin{equation*}
  d_{TV}(\PP_{f,g},\PP_{\tilde{f},g}) \leq \frac{\|f-\tilde{f}\|}{\sqrt{2}}.
 \end{equation*}
\end{lemma}
The next proposition is concerned by the closeness of two laws $\PP_{f,g}$ and $\PP_{f,\tilde{g}}$, when we keep the same shape $f \in \cH_1$. Consider the inverse functions of the distribution functions defined by
\[ \forall u \in [0, 1], \quad G^{-1}(u) = \inf \{t \in [0, 1] : g([0, t]) > u \}. \]
The Wasserstein (or Kantorovich) distance $W_1$ is given by
\[ W_1(g, \tilde{g}) := \int_0^1 \left| G^{-1}(t) - \tilde{G}^{-1}(u) \right| d t. \]
\begin{proposition}\label{prop:transport}
 Consider $f \in \cH_1$, and let $g$ and $\tilde{g}$ be any measures on $[0,1]$. Then
 \begin{align*}
  d_{TV}(\PP_{f,g},\PP_{f,\tilde{g}}) &\leq \sqrt{2} \pi \|f\|_{\cH_1} W_1(g, \tilde{g}) \\ &
  \leq \sqrt{2} \pi \|f\|_{\cH_1} d_{TV}(g, \tilde{g}) 
  \leq  \pi \|f\|_{\cH_1} \| g - \tilde{g} \|/\sqrt{2}.
 \end{align*}
\end{proposition}
The last two upper bounds are useful in our setting because we only consider distributions that admit regular densities.

\begin{proof}[Proof of Proposition \ref{prop:transport}]
 We use a change of variable, the convexity of $d_{TV}$, and Lemma \ref{lemma:dVT_sur_f} to get
 \begin{align*}
  d_{TV}(\PP_{f,g},\PP_{f,\tilde{g}}) &= \left\| \int_0^1 \PP_{f,\delta_\alpha} d g(\alpha) - \int_0^1 \PP_{f,\delta_\alpha} d \tilde{g}(\alpha) \right\|_{TV} \\
  &= \left\| \int_0^1 \left(\PP_{f,\delta_{G^{-1}(u)}} - \PP_{f,\delta_{\tilde{G}^{-1}(u)}} \right) d u \right\|_{TV} \\
  &\leq \int_0^1 d_{TV}\left(\PP_{f^{-G^{-1}(u)},\delta_0}, \PP_{f^{-\tilde{G}^{-1}(u)},\delta_0} \right) d u \\
  &\leq \frac{1}{\sqrt{2}} \int_0^1 \left\|f^{-G^{-1}(u)} - f^{-\tilde{G}^{-1}(u)} \right\| d u.
 \end{align*}
 Then
 \begin{align*}
  \left\|f^{-G^{-1}(u)} - f^{-\tilde{G}^{-1}(u)} \right\| &= \sqrt{\sum_{k\in\Z} |c_k(f)|^2 \left| e^{-\i 2 \pi k G^{-1}(u)} - e^{-\i 2 \pi k \tilde{G}^{-1}(u)} \right|^2 } \\
  &\leq 2 \pi \left| G^{-1}(u) - \tilde{G}^{-1}(u) \right| \sqrt{\sum_{k\in\Z} k^2 |c_k(f)|^2 }.
 \end{align*}
 Therefore we get the first inequality:
 \[ d_{TV}(\PP_{f,g},\PP_{f,\tilde{g}}) \leq \sqrt{2} \pi \|f\|_{\cH_1} \int_0^1 \left| G^{-1}(u) - \tilde{G}^{-1}(u) \right| d u. \]
 
 Now, the second inequality is a classical result: see for instance \cite[Theorem 4]{GS02}. The last inequality is well known too.
\end{proof}

\begin{proof}[Proof of Theorem  \ref{theo:posterior_shift}]
We mimic the proof of Theorem 2.2 of \cite{BG_1}.

\paragraph{Complementary of the sieve} 
 First, we consider the sieve over $\mathcal{P}$ defined as the set of all possible laws when $f$ has truncated Fourier coefficients and a restricted $L^2$ norm:
$$
\cP_{k_n,w_n} := \left\{\PP_{f,g} : (f ,g) \in \cF^{k_n} \times \mathfrak{M}_{\nu}([0,1])(2A),
 \|f\| \leq w_n\right\},
$$
where $k_n$ is a sequence such that $k_n \longmapsto + \infty$ as $n \longmapsto + \infty$, and $w_n^2= 4 k_n+2$.

Since our sieve is included in the set of all mixture laws, we can apply  Proposition 3.10 of \cite{BG_1} and get
$$
\Pi_n \left(\cP \setminus \cP_{k_n,w_n}\right)  \leq  e^{ 
- c [ k_n^2 \log^\rho(k_n) \wedge k_n \xikn^{-2}]}.
$$

\paragraph{Entropy estimates} Since $\mathfrak{M}_{\nu}([0,1])(2A) \subset \M$, our sieve is included in the sieve considered in \cite{BG_1}, we also deduce that for any sequence $\epsilon_n \longmapsto 0$:
$$
\log D\left( \epsilon_n, \cP_{k_n,w_n},d_H \right) \lesssim  k_n^2 \left[ \log k_n + \log \frac{1}{\epsilon_n} \right].
$$

\paragraph{Lower bound of the prior of Kullback neigbourhoods}
We use the description of Kullback neigbourhoods based on our preliminary results .  We define $\tilde{\epsilon}_n = c\epsilon_n \left( \log  \frac{1}{\epsilon_n}\right)^{-1}$, an integer
 $\ell_n$ such that $ \tilde{\epsilon}_n^{-1/s} \lesssim  \ell_n \lesssim \tilde{\epsilon}_n^{-1/s}$, and the sets
$$
\cF_{\tilde{\epsilon}_n} := \left\{ f \in \cH_s^{\ell_n} : \|f-f^0_{\ell_n}\| \leq \tilde{\epsilon}_n^2 \right\},
$$
and 
$$
\cG_{\tilde{\epsilon}_n} := \left\{ g \in \mathfrak{M}_{\nu}([0,1])(2A): d_{TV}(g,g^0) \leq \tilde{\epsilon}_n \right\}.
$$
We deduce from Lemma \ref{lemma:dVT_sur_f}, Proposition \ref{prop:transport} and arguments of Proposition 3.9 of \cite{BG_1} that as soon as $f \in \cF_{\tilde{\epsilon}_n}$ and $g \in \cG_{\tilde{\epsilon}_n}$, $\PP_{f,g}$ belongs to an $\epsilon_n$ Kullback neighbourhood of $\PP_{f^0,g^0}$.
From Proposition 3.9, we can use the following lower bound of the prior mass on $\cF_{\tilde{\epsilon}_n}$:
$$
\Pi_n \left( \cF_{\tilde{\epsilon}_n} \right) \geq  e^{-(c +o(1))\, \left[ \epsilon_n^{-2/s} \left(\log (1/\epsilon_n)\right)^{\rho+2/s} \vee \xikn^{-2} \right] }.
$$
According to Theorem \ref{theo:lower_bound_proba} given in the appendix, we know that 
$$
\Pi_n \left( \cG_{\tilde{\epsilon}_n} \right) \geq e^{-(c+o(1)) \tilde{\epsilon}_n^{-1/(k_\nu+1/2)}} \geq  e^{-(c+o(1)) \tilde{\epsilon}_n^{-\frac{1}{\nu}}}$$
since $k_{\nu}+1/2 \leq \nu$ (see also \cite{Li_Shao} for a very complete survey on the small ball probability estimation for Gaussian processes).

\paragraph{Contraction Rate}
We now find a suitable choice of $k_n$ and $\epsilon_n$ in order to satisfy Theorem 2.1 of \cite{GGvdW00}, \textit{i.e.}
$$
\Pi_n\left( \cG_{\tilde{\epsilon}_n} \right) \Pi_n\left( \cF_{\tilde{\epsilon}_n} \right) \geq e^{- C n \epsilon_n^2}
$$
$$
\log D\left( \epsilon_n, \cP_{k_n,w_n},d_H \right) \lesssim   n \epsilon_n^2
$$
$$
 \Pi_n \left(\cP \setminus \cP_{k_n,w_n}\right)  \leq e^{-(C+4)n \epsilon_n^2}.
$$
Following the arguments already developed in Theorem 2.2 of \cite{BG_1}, we can find $\gamma>0$ and $\kappa>0$ such that 
$$
\epsilon_n := n^{- \left[ \frac{\nu}{2\nu+1} \wedge \frac{s}{2s+2} \wedge \frac{3}{8}\right] } \log (n) ^{\kappa}, \qquad k_n = n^{\frac{1}{2} - \left[\frac{\nu}{2\nu+1} \wedge \frac{s}{2s+2} \wedge \frac{3}{8}\right]} \log(n)^{\gamma}.
$$
\end{proof}

\section{Identifiability and semiparametric results} \label{sec:identifiability}

In the Shape Invariant Model, an important issue is the identifiability of the model with respect to the unknown curve $f$ and the unknown mixture law $g$. We first discuss on a quite generic identifiability condition for $\PP_{f,g}$. Then, we deduce from Theorem \ref{theo:posterior_shift} a contraction rate of the posterior distribution around the true $f^0$ and $g^0$.

\subsection{Identifiability of the model}

In previous works on SIM, the identifiability of the model is generally given according to a restriction on the support of $g$. For instance, \cite{BG10} assume the support of $g$ to be an interval included in $[-1/4,1/4]$ (their shapes are defined on $[-1/2;1/2]$ instead of $[0,1]$ in our paper) and $g$ is assumed to have $0$ mean although $f$ is supposed to have a non vanishing first Fourier coefficient ($\theta_1(f) \neq 0$). The same kind of condition on the support of $g$ is also assumed in \cite{BG12}.

If the condition on the first harmonic on $f$ is imperative to obtain identifiability of $g$, the restriction on its support size seems artificial and we detail in the sequel how one can avoid such a hypothesis.
First, we recall that for any curve $Y$ sampled from the SIM, the first Fourier coefficient is given by $\theta_1(Y) = \theta^0_1 e^{-\i 2 \pi \tau} + \xi$ (here $\theta_1^0 = \theta_1(f^0)$). Hence, up to  a simple change of variable in $\tau$, we can always modify $g$ in $\tilde{g}$ such that 
$\theta_1^0 \in \R_{+}$. 
It is for instance sufficient to fix $\tilde{g}(\varphi) = g(\varphi+\alpha)$ where $\alpha$ is the complex argument of $\theta^0_1$. Hence, to impose such an identifiability condition, we have chosen to restrict $f$ to $\cF_s$. 
This condition is not restrictive up to a change of measure for the random variable $\tau$.  We now establish the proof of Theorem \ref{theo:ident}.

\begin{proof}[Proof of Theorem  \ref{theo:ident}]
The demonstration is decomposed using three hierarchical steps. First, we prove that if $ \PP_{f,g} = \PP_{\tilde{f},\tilde{g}}$, then one has necessarily $\theta_1(f) = \theta_1(\tilde{f})$. Then we deduce from this point that $g = \tilde{g}$ and at last we obtain the identifiability for all other Fourier coefficients of $f$.

Note that as soon as $\nu>1/2$, $g$ and $\tilde{g}$ admit densities w.r.t. the Lebesgue measure on $[0, 1]$. In the sequel we use the same notation $g$ to refer to the density of $g$.

\paragraph{Point 1: Identifiability on $\theta_0(f)$ and $\theta_1(f)$}
We denote $\PP^k_{f,g}$ the marginal law of $\PP_{f,g}$ on the $k^{th}$  Fourier coefficient when the curve follows the Shape Invariant Model (\ref{eq:model_fourier}). Of course, we have the following implications
$$
 d_{TV}(\PP_{f,g},\PP_{\tilde{f},\tilde{g}}) = 0 \Longrightarrow \left( \PP_{f,g} = \PP_{\tilde{f},\tilde{g}} \right) \Longrightarrow  \forall k \in \Z : d_{TV}\left(\PP^k_{f,g},\PP^k_{\tilde{f},\tilde{g}}\right) = 0.
$$
We immediately obtain that $\theta_0(f) = \theta_0(\tilde{f})$ since $\theta_0(f)$ (resp. $\theta_0(\tilde{f})$) represents the mean of the distribution $\PP^0_{f,g}$ (resp. $\PP^0_{\tilde{f},\tilde{g}}$). But note that the distribution $\PP^0_{f,g}$ does not bring any information on the measure $g$, and is not helpful for its identifiability.
Concerning now the first Fourier coefficient, we use the notation $\theta_1:=\theta_1(f)$, $\tilde{\theta}_1: = \theta_1(\tilde{f})$ and remark that
\begin{multline*}
 d_{TV}\left(\PP^1_{f,g},\PP^1_{\tilde{f},\tilde{g}}\right) \\ = \frac{1}{2\pi}
\int_{\C} \left|\int_{0}^1 e^{-|\theta_1 e^{\i 2 \pi \alpha} - z|^2} g(\alpha) d\alpha - \int_{0}^1 e^{- |\tilde{\theta}_1 e^{\i 2 \pi \alpha} - z|^2 } \tilde{g}(\alpha) d\alpha  \right| dz.
\end{multline*}
Assume now that $\tilde{\theta}_1 \neq \theta_1$, without loss of generality $\tilde{\theta}_1 > \theta_1 >0$ and consider the disk  $D_{\C}\left(0,\frac{\tilde{\theta}_1-\theta_1}{2}\right)$, we then get
$
\forall z \in D_{\C}\left(0,\frac{\tilde{\theta}_1-\theta_1}{2}\right)\,, \forall \alpha \in [0,1]:$
$$  |\theta_1 e^{i 2 \pi \alpha} - z | < \frac{\tilde{\theta}_1+\theta_1}{2} \, \text{and} \, |\tilde{\theta}_1 e^{i 2 \pi \alpha} - z | > \frac{\tilde{\theta}_1+\theta_1}{2} . $$
Hence, for all $z \in D_{\C}\left(0,\frac{\tilde{\theta}_1-\theta_1}{2}\right)$, we get 
$\int_{0}^1 e^{-|\theta_1 e^{\i 2 \pi \alpha} - z|^2} g(\alpha) d\alpha > e^{- \frac{\left|\tilde{\theta}_1+\theta_1\right|^2}{4} } $
and of course 
$\int_{0}^1 e^{-|\tilde{\theta}_1 e^{\i 2 \pi \alpha} - z|^2} \tilde{g}(\alpha) d\alpha < e^{- \frac{\left|\tilde{\theta}_1+\theta_1\right|^2}{4} }.
$
We can thus write the following lower bound of the Total Variation
\begin{multline*}
 d_{TV}\left(\PP^1_{f,g},\PP^1_{\tilde{f},\tilde{g}}\right) \geq \frac{1}{2\pi}\int_{D_{\C}\left(0,\frac{\tilde{\theta}_1-\theta_1}{2}\right)} \left|\int_{0}^1 e^{-|\theta_1 e^{\i 2 \pi \alpha} - z|^2} g(\alpha) d\alpha \right. \\ 
 - \left. \int_{0}^1 e^{- |\tilde{\theta}_1 e^{\i 2 \pi \alpha} - z|^2 } \tilde{g}(\alpha) d\alpha\right|dz > 0.
\end{multline*}
In the opposite,  $d_{TV}(\PP^1_{f,g},\PP^1_{\tilde{f},\tilde{g}})  = 0$ implies that $\theta_1=\tilde{\theta}_1$ since $f$ and $\tilde{f}$ belong to $\cF_s(A)$.

\paragraph{Point 2: Identifiability on $g$}
We still assume that  $d_{TV}(\PP^1_{f,g},\PP^1_{\tilde{f},\tilde{g}})  = 0$. We know that  $\theta_1=\tilde{\theta}_1$ and we want to infer that $g= \tilde{g}$. We are going to establish this result using only the first harmonic of the curves.
Using a polar change of variables $z=\rho e^{\i \varphi}$, we can write that
\begin{multline*}
 d_{TV}\left(\PP^1_{f,g},\PP^1_{\tilde{f},\tilde{g}}\right) \\
 \begin{aligned}
  &= \frac{1}{2\pi} \int_{\C} e^{-[\theta_1^2 + |z|^2]} \left| \int_{0}^1 e^{2 \re (z \theta_1 e^{\i 2 \pi \alpha})} (g(\alpha)-\tilde{g}(\alpha) d\alpha \right| dz \\
 & = \frac{1}{4 \pi^2}\int_{0}^{+ \infty} \rho e^{-[\theta_1^2 + \rho^2]} \int_{0}^{2 \pi} \left| \int_{0}^{2 \pi} e^{2 \rho \theta_1 \cos(u - \varphi)}  (g-\tilde{g})(u/2\pi) du\right| d\varphi d\rho \\
 & = \frac{1}{4 \pi^2}\int_{0}^{+ \infty} \rho e^{- [\theta_1^2 + \rho^2]} \int_{0}^{2 \pi} \left| \int_{0}^{2 \pi} e^{2 \rho \theta_1 \cos(u )}  (g-\tilde{g})\left(\frac{u + \varphi}{2\pi}\right) d\alpha\right| d\varphi d\rho  \\
 & =   \frac{1}{4 \pi^2}\int_{0}^{+ \infty} \rho e^{- [\theta_1^2 + \rho^2]} \int_{0}^{2 \pi} \left| \psi_{2 \rho \theta_1}(\varphi) \right| d \varphi d\rho.
 \end{aligned}
\end{multline*}
 In the expression above, we denote $h = g-\tilde{g}$ and $\psi_{a}(\varphi)$ is defined as  
  $$
  \psi_{a}(\varphi) = \int_{0}^{2 \pi} e^{a  \cos (u)} h\left(\frac{u + \varphi}{2\pi}\right) d u.
  $$
  Of course, $\psi_a$ is upper bounded by $4 \pi e ^a$, and a very rough inequality yields
  $|    \psi_{a}(\varphi) |\geq \frac{|\psi_{a}(\varphi) |^2}{4 \pi e^a}$. Hence, 
  \begin{equation}\label{eq:dvt_g_psi}
  d_{TV}(\PP^1_{f,g},\PP^1_{\tilde{f},\tilde{g}}) \geq \frac{1}{8\pi^2} \int_{0}^{+ \infty} \rho e^{- (\theta_1^2 + \rho^2 + 2 \theta_1 \rho)} \|\psi_{2 \rho \theta_1}\|^2 d \rho.
  \end{equation}
Using the fact that $\nu > 1$, $h$ may be expanded in Fourier series since  $h \in \cL^2([0,1])$:
 $$
 h(x) = \sum_{n \in \Z} c_n(h) e^{i 2 \pi n x},
  $$
and we can also obtain the Fourier decomposition of  $\psi_a$:
  $$
  \psi_a(\varphi)= \sum_{n \in \Z} c_n(h) \int_{0}^{2 \pi} e^{a \cos(u)} e^{i n u} du\, e^{i 2 \pi n \varphi}.
  $$
Thus, the $L^2$ norm of $\psi_a$ is given by
  \begin{equation}\label{eq:norm2psi}
  \|\psi_{a}\|^2 = \sum_{n \in \Z} |c_n(h)|^2 \left| \int_{0}^{2 \pi} e^{a \cos(u)} e^{\i nu} du\right|^2.
  \end{equation}
 Now, if we denote the first and second kind of Tchebychev polynomials $(T_n)_{n \in \Z}$ and $(U_n)_{n \in \Z}$ which satisfy $T_n(\cos \theta) = \cos (n \theta)$ and $(\sin \theta) U_n(\cos \theta) = \sin (n \theta)$, we can decompose
 \begin{multline*}
 \int_{0}^{2 \pi} e^{a \cos(u)} e^{\i nu} du \\
 \begin{aligned}
  &= \int_{0}^{2 \pi} e^{a \cos(u)} \left[ T_n(\cos u) + \i (\sin u) U_n(\cos u)\right] d u \\
  & = \int_{0}^{2 \pi} \sum_{k \geq 0} \frac{a^k (\cos u)^k}{k!} \left[ T_n(\cos u) + \i (\sin u) \sum_{j=0}^{n} \beta_j (\cos u)^j\right] d u
 \end{aligned}
\end{multline*}
where we have used the analytic expression of $U_n$ given by
$$U_n(\cos u) = \sum_{j=0}^{E((n-1)/2)} (-1)^j C_{n}^{2j+1} (\cos u)^{n-2j-1} (1 - \cos^2 u)^j.$$
 Hence, we obtain
\begin{align*}
 \int_{0}^{2 \pi} e^{a \cos(u)} e^{\i nu} du &=  \int_{0}^{2 \pi}  \sum_{k \geq 0} \frac{a^k (\cos u)^k}{k!}  T_n(\cos u)  d u  \\ &\quad +  \i \sum_{k \geq 0}\sum_{j=0}^n \beta_j \frac{a^k}{k!}  \int_{0}^{2 \pi}   \sin u (\cos u)^{k+j} d u \\
 & =  \int_{0}^{2 \pi}  \sum_{k \geq 0} \frac{a^k (\cos u)^k}{k!}  T_n(\cos u)  d u \\
 &  =  \int_{0}^{2 \pi} e^{a \cos(u)}  \cos (n u) d u \in \R \quad \text{if}\quad a\in\R.
 \end{align*}
We denote $A_n$ the following (holomorphic) function of the variable $a$ as
$$
  A_n(a):= \int_{0}^{2 \pi} e^{a \cos(u)} \cos(n u) du,$$
  and equation (\ref{eq:norm2psi}) yields
  \begin{equation}\label{eq:psi2normexplicit}
  \|\psi_a\|^2 = \sum_{n \in \Z} |c_n(h)|^2 A_n(a)^2.
  \end{equation}
Moreover, for each $n$,  $A_n$ is not the null function, otherwise it would be the case for each of its derivative but remark that $(\cos u)^n$ may be decomposed in the basis $(T_k)$ and using successive derivations
\begin{align*}
 A_n^{(n)}(0) &=\frac{d^{(n)}}{da^{(n)}} \left[ \sum_{k=0}^\infty \frac{a^k}{k!} \int_0^{2\pi} (\cos u)^k \cos(n u) \, du \right] (0) \\
 &=\int_{0}^{2 \pi} (\cos u)^n T_n(\cos u) d u \\
 & =  \int_{0}^{2 \pi} \left[ \sum_{k=0}^{n-1} \alpha_k T_k(\cos u) + 2^{1-n} T_n(\cos u) \right] T_n(\cos u) d u \\
 &= 2^{1-n} \pi >0.
\end{align*}
Note that in the meantime, we also obtain that $A_n^{(j)}(0)=0, \forall j < n$, so that
\begin{equation}\label{eq:Ansize}
A_n(a) \sim_{a \mapsto 0} \frac{2^{1-n} \pi}{n!} a^n.
\end{equation}
We can conclude the proof of the identifiability of $g$ using \eqref{eq:psi2normexplicit} in \eqref{eq:dvt_g_psi} to obtain
$$
  d_{TV}(\PP^1_{f,g},\PP^1_{\tilde{f},\tilde{g}}) \geq \frac{1}{8\pi^2} 
 \sum_{n \in \Z} |c_n(h)|^2  \underbrace{\left(\int_{0}^{+ \infty} \rho e^{- [\theta_1+\rho]^2} A_n(2 \rho \theta_1)^2 d \rho\right)}_{:=I_n(\theta_1)}.
  $$
  From  \eqref{eq:Ansize}, we can deduce that each integral $I_n(\theta_1) \neq 0, \forall n \in \Z$ and we then conclude that: 
  $$ d_{TV}(\PP^1_{f,g},\PP^1_{\tilde{f},\tilde{g}})  \Longleftrightarrow g=\tilde{g} \qquad \text{et} \qquad \theta_1=\tilde{\theta}_1.$$
  
  \paragraph{Point 3: Identifiability on $f$} 
  We end the argument and prove that $\PP_{f,g}  = \PP_{\tilde{f},\tilde{g}}$ implies $f=\tilde{f}$. We already know that $g=\tilde{g}$ and it remains to establish the equality for all the Fourier coefficients whose frequency is different from $0$ and $1$. By a similar argument as the one used for the identifiability of $\theta_1$ (Point 1), we can easily show that 
$$   
d_{TV}(\PP^k_{f,g},\PP^k_{\tilde{f},\tilde{g}})=0 \Longrightarrow |\theta_k| = |\tilde{\theta}_k|.
$$
But we cannot directly conclude here since it is not reasonable to restrict the phase of each others coefficients $\theta_k(f)$ to a special value (as it is the case for $\theta_1(f)$ which is positive). 
We assume that  $\tilde{\theta}_k= \theta_k e^{\i \varphi}$.  Since $g=\tilde{g}$, we have
$$   
d_{TV}(\PP^k_{f,g},\PP^k_{\tilde{f},g})=\frac{1}{2\pi} \int_{\C} \underbraceabs{\int_{0}^{2\pi} e^{- | z - \theta_k e^{- \i k \alpha} |^2} - e^{-| z - \theta_k e^{\i (\varphi- k\alpha)}|^2} g(\alpha) d\alpha }{:=F(z)} d z.$$
Now, if one considers $z=x+\i y$, $F$ is differentiable with respect to $x$ and $y$ and $F(0)=0$. A simple computation of $\nabla F(0)$ shows that $\nabla F(0)$ is the  vector (written in the complex form)
$$
\nabla F(0) = \theta_k e^{- |\theta_k|^2} c_k(g) [1-e^{\i \varphi}].
$$
Since $g \in \M^{\star}$,  this last term is non vanishing except if $\theta_k=0$ (which trivially implies that $\tilde{\theta}_k=0=\theta_k$) or if $\varphi\equiv 0 (2\pi)$. In both cases, $F'(0) = 0 \Longleftrightarrow  \tilde{\theta}_k = \theta_k$. Thus, as soon as
$\theta_k \neq \tilde{\theta}_k$, we have $\nabla F(0) \neq 0$ and we may find a neighbourhood of $0$ denoted $B(0,r)$ such that $|F|(z) >0$ when $z \in B(0,r)\setminus \{0\}$ . 
  This is sufficient to end the proof of identifiability.
\end{proof}


In a sense, the main difficulty of the proof above is the implication of $d_{TV}(\PP^1_{f,g},\PP^1_{\tilde{f},\tilde{g}})  \Longrightarrow g=\tilde{g}$. Then, the identifiability follows using a chaining argument $\theta_1(f) \rightarrow g \rightarrow \theta_k(f), \forall k \notin \{0,1\}$. We will see that this part of the proof can also be used to obtain a contraction rate for $f$ and $g$ around $f^0$ and $g^0$. We recall here the main inequality used above: $\forall \theta_1>0$ and $ \forall (g,\tilde{g}) \in \Mnu$, the identifiability on $g$ is traduced by

\begin{equation}\label{eq:main_g}
d_{TV} \left(\PP^1_{\theta_1,g},\PP^1_{\theta_1,\tilde{g}} \right) 
\geq \frac{1}{8 \pi^2} 
\sum_{n\in\Z} |c_n(g-\tilde{g})|^2 
\left( \int_{0}^{\infty} \rho e^{-(\rho+\theta_1)^2} A_n(2 \rho \theta_1)^2 d \rho \right)
\end{equation}
The aim of the next paragraph is to exploit this inequality to produce a contraction rate of $g$ aroung $g^0$.

\subsection[Posterior contraction rate around f and g]{Contraction rate of the posterior distribution around $f^0$ and $g^0$}

\subsubsection{Link with deconvolution with unknown variance operator}
We  provide in this section an upper bound on the contraction rate of the posterior law around $f^0$ and $g^0$. This question is somewhat natural owing to the identifiability result obtained in the previous section. We thus assume for the rest of the paper that $f \in \cF_s$ and $g \in \Mnu$ for some parameters $s \geq  1$ and $\nu > 1$.

Remark first that our problem written in the Fourier domain seems strongly related to the standard deconvolution with unknown variance setting. For instance, the first observable Fourier coefficients are
$$
\theta_1(Y_j) = \theta_1 e^{-\i 2 \pi \tau_j} + \epsilon_{1,j}, \forall j \in \{1 \ldots n\}
$$
and up to a division by $\theta_1$, it can also be parametrised as
\begin{equation}\label{eq:deconvolution_unknown}
\tilde{\theta}_1(Y_j) = e^{-\i 2 \pi \tau_j} + \frac{\epsilon_{1,j}}{\theta_1}, \forall j \in \{1 \ldots n\},
\end{equation}
which is very similar to the problem $Y=X+\epsilon$ studied for instance by \cite{M02} where $\epsilon$ follows a Gaussian law whose variance (here $1/\theta_1^2$) is unknown. As pointed in \cite{M02} (see also the more recent work \cite{BM05} where similar situations are extensively detailed), such a particular setting is rather unfavourable for statistical estimation since  convergence rates are generally of logarithmic order. Such a phenomenon also occurs in our setting, except for the first Fourier coefficient of $f$ as pointed in the next proposition.

The roadmap of this paragraph is similar to the proof of Theorem \ref{theo:ident}. We first provide a simple lower bound of $d_{TV}$ which enables to conclude for the first Fourier coefficient. Then, we still use the first marginal to compute a contraction rate for the posterior distribution on $g$ around $g^0$. At last, we chain all these results to provide a contraction rate for the posterior distribution on $f$ around $f^0$.

\subsubsection{Contraction rate on the first Fourier coefficient}

{\sloppy
\begin{proposition}\label{prop:theta1}
Assume that $(f,g) \in \cF_s \times \Mnu$, then the posterior distribution satisfies
$$
\Pi_n \left( \left.\theta_1 \in B\left(\theta_1^0,M \epsilon_n^{1/3}\right)^c \right\vert Y_1, \ldots, Y_n \right) \mapsto 0
$$ in $\PP_{f^0,g^0}$ probability as $n \rightarrow + \infty$ for a sufficiently large $M$. The contraction rate around the true Fourier coefficient is thus at least $n^{-1/3 \times [\nu/(2\nu+1) \wedge s /(2s+2) \wedge 3/8]} (\log n)^{1/3}$.
\end{proposition}
\fussy}

\begin{proof}
The demonstration is quite simple. Remark that using the beginning of the proof of Theorem \ref{theo:ident}, one can show that for any $\theta_1$ such that $0<\eta<|\theta_1-\theta_1^0| <\theta_1^0/2$, one can bound, for any $g \in \Mnu$, the Total Variation distance between $\PP_{f,g}$ and $\PP_{f^0,g^0}$. Remark that
$$
d_{TV}\left(\PP_{f,g},\PP_{f^0,g^0} \right) \geq d_{TV}\left(\PP^1_{f,g},\PP^1_{f^0,g^0} \right),
$$
owing to the restriction of $\PP_{f,g}$ to the first Fourier marginal and the variational definition of the Total Variation distance. Then 
\begin{multline*}
 d_{TV}\left(\PP^1_{f,g},\PP^1_{f^0,g^0} \right) \\
 \begin{aligned}
  &\geq \frac{1}{2\pi} \int_{B\left(0,\frac{|\theta_1-\theta_1^0|}{4}\right)} \left|\int_{0}^1 g(\alpha) e^{-|z-\theta_1e^{\i 2 \pi \varphi}|^2} -  g^0(\alpha) e^{-|z-\theta^0_1e^{\i 2 \pi \varphi}|^2} d \varphi \right| d z \\
  & \geq \frac{\eta^2}{32}\left| e^{-(3 \theta_1^0+\theta_1)^2/16} - e^{-(3 \theta_1+\theta^0_1)^2/16}\right| \geq C(\theta_1^0) \eta^3,
 \end{aligned}
\end{multline*}
for a suitable small enough constant $C(\theta_1^0)$. Now, one can use simple inclusions and Pinsker inequality 
\begin{multline*}
 \left\{ \theta_1 \in B(0,\eta)^c \right\} \subset \left\{ \theta_1 \vert d_{TV}(\PP_{f,g},\PP_{f^0,g^0}) \geq C(\theta_1^0) \eta^3 \right\} \\ \subset 
\left\{ \theta_1 \vert d_{H}(\PP_{f,g},\PP_{f^0,g^0}) \geq C(\theta_1^0) \eta^3 \right\}.
\end{multline*}
The proof is now achieved according to Theorem \ref{theo:posterior_shift}.
\end{proof}

\subsubsection[Posterior contraction rate around g0]{Posterior contraction rate around $g^0$}

We now study the contraction rate of the posterior distribution around the true mixture law $g^0$. This result is stated below.

\begin{theo}\label{theo:contraction_g}
Assume $(f^0,g^0)\in \cF_s \times \Mnu$, then
$$
\Pi_n\left( \left. g : \|g-g^0\|^2 > M \log^{- 2 \nu}(n) \right\vert Y_1, \ldots, Y_n \right) \longrightarrow 0
$$ in $\PP_{f^0,g^0}$ probability as $n \rightarrow + \infty$ for a sufficiently large $M$.
\end{theo}

\begin{proof}
We first restrict ourselves to the first marginal on Fourier coefficient as before. Using Theorem 2.2, we know that
$$
\Pi_n \left\{ \PP_{f,g} \quad s.t.\,  d_H(\PP_{f,g} ,\PP_{f^0,g^0}) \geq M \epsilon_n \vert Y_1, \ldots Y_n \right\} \longrightarrow 0 \quad \text{as}  \quad  n \rightarrow + \infty. 
$$
Since $$d_{TV}(\PP^1_{\theta_1,g} ,\PP^1_{\theta_1^0,g^0}) = d_{TV}(\PP^1_{f,g} ,\PP^1_{f^0,g^0})\leq d_{TV}(\PP_{f,g} ,\PP_{f^0,g^0})\leq d_H(\PP_{f,g} ,\PP_{f^0,g^0}),$$ we then get 
\begin{equation}\label{eq:tronc0}
\Pi_n \left\{ \PP_{f,g} \quad s.t. \, d_{TV}(\PP^1_{\theta_1,g} ,\PP^1_{\theta_1^0,g^0}) \geq M \epsilon_n \vert Y_1, \ldots Y_n \right\} \longrightarrow 0 \quad \text{as} \quad n \rightarrow + \infty. \end{equation}

For any $g \in \Mnu$, the triangular inequality yields
\begin{equation}\label{eq:tronc1}
 d_{TV}\left(\PP^1_{\theta^0_1,g},\PP^1_{\theta_1,g}\right) + 
d_{TV} \left(\PP^1_{\theta_1,g},\PP^1_{\theta_1^0,g^0}\right) \geq d_{TV} \left(\PP^1_{\theta^0_1,g},\PP^1_{\theta_1^0,g^0}\right) .
\end{equation}

Now, let $\tilde{f}$ be defined by $\theta_1(\tilde{f}) = \theta_1(f)$, and for any $k\in\Z\backslash\{1\}$, $\theta_k(\tilde{f}) = \theta_k(f^0)$. Then Lemma \ref{lemma:dVT_sur_f} yields
\[ d_{TV}\left(\PP^1_{\theta^0_1,g},\PP^1_{\theta_1,g}\right) = d_{TV}\left(\PP^1_{\tilde{f},g},\PP^1_{f^0,g}\right) \leq \frac{\|\tilde{f}-f^0\|}{\sqrt{2}} = \frac{|\theta_1-\theta_1^0|}{\sqrt{2}}. \]
Therefore
\begin{multline}\label{eq:tronc2}
\Pi_n \left( \left. \PP_{f,g} \, s.t. \, d_{TV}\left(\PP^1_{\theta^0_1,g},\PP^1_{\theta_1,g} \right) \leq \frac{M}{\sqrt{2}} \epsilon_n^{1/3} \right\vert Y_1, \ldots, Y_n\right) \\ 
\geq \Pi_n \left(\left. \PP_{f,g} \, s.t. \, |\theta_1-\theta_1^0| \leq M \epsilon_n^{1/3}\right\vert Y_1, \ldots, Y_n \right) \longrightarrow 1 
\end{multline}
as $n \rightarrow + \infty$. 
In conclusion, we deduce from \eqref{eq:tronc0},\eqref{eq:tronc1} and \eqref{eq:tronc2} that for $M$ large enough:
$$\Pi_n \left(\left. \PP_{f,g} \, s.t. \, d_{TV}\left(\PP^1_{\theta^0_1,g},\PP^1_{\theta^0_1,g^0}\right) \leq M \epsilon_n^{1/3}\right\vert Y_1, \ldots, Y_n \right) \longrightarrow 1 \quad \text{as} \quad n \rightarrow + \infty.
$$

We then use equation \eqref{eq:main_g} applied with $\theta_1=\theta_1^0$ and the last equation to obtain our rate of consistency. Remark that
  \begin{equation}\label{eq:dvtlower}
  d_{TV}(\PP^1_{\theta^0_1,g},\PP^1_{\theta^0_1,g^0}) \geq \frac{1}{8\pi^2 } 
 \sum_{n \in \Z} |c_n(g-g^0)|^2  \int_{0}^{+ \infty} \rho e^{- (\rho+\theta_1^0)^2} A_n(2 \rho \theta_1^0)^2 d \rho,
 \end{equation}
where we have used the definition
$$
A_n(a) = \int_{0}^{2\pi}e^{a \cos(u)} \cos (nu) d u.
$$
Now, we use equivalents given by Lemma \ref{lemma:equi_bessel} detailed 
in the Appendix. We only keep the integral of $A_n$ for $a \in [0,c\sqrt{n}]$ since it can be shown that the tail of such integral will yield neglictible term 
We just use the equivalent given by \eqref{eq:equi2}.

One can find a sufficiently small constant $\kappa$ such that
\begin{multline*}
 \int_{0}^{+ \infty} \rho e^{- (\rho+\theta_1^0)^2} A_n(2 \rho \theta_1^0)^2 d \rho \\
 \begin{aligned}
  & \geq \int_{0}^{\frac{\sqrt{n}}{2 \theta_1^0}}  \frac{4 \pi^2 \rho^{2n+1} \{\theta_1^0\}^{2n}}{n!^2}e^{-(\rho+\theta_1^0)^2} \left(1-\kappa\frac{ [2 \rho \theta_1^0]}{n}\right)^2 d \rho \\
  & \geq \left(1-\frac{ \kappa}{\sqrt{n}}\right)^2 \frac{4 \pi^2 \{\theta_1^0\}^{2n}}{n!^2} e^{-\big(\theta_1^0+\frac{\sqrt{n}}{2\theta_1^0}\big)^2}  \int_{0}^{\frac{\sqrt{n}}{2\theta_1^0}} \rho^{2n+1} d \rho
 \end{aligned}
\end{multline*}
Now, we can apply the Stirling formula to obtain:
\begin{multline*}
 \frac{4 \pi^2 \{\theta_1^0\}^{2n}}{n!^2} e^{-\big(\theta_1^0+\frac{\sqrt{n}}{2\theta_1^0}\big)^2}  \int_{0}^{\frac{\sqrt{n}}{2\theta_1^0}} \rho^{2n+1} d \rho \\
 \begin{aligned}
  &\sim \frac{4 \pi^2 \{\theta_1^0\}^{2n}}{(n/e)^{2n} 2 \pi n}e^{-\big(\theta_1^0+\frac{\sqrt{n}}{2\theta_1^0}\big)^2} \frac{\left( \sqrt{n}/(2 \theta_1^0)\right)^{2n+2}}{2n+2} \\
  & \sim  \frac{2 \pi}{n(2n+2)} e^{ - 2n \log \left[ \frac{n}{e \theta_1^0}\right]-\big(\theta_1^0+\frac{\sqrt{n}}{2\theta_1^0}\big)^2+(n+1)\log \left[\frac{n}{4 \{\theta_1^0\}^2}\right]}.
 \end{aligned}
\end{multline*}
Hence, this last term is lower bounded by $C(\theta_1^0) e^{- n \log(n)}$.
  As a consequence, we can plug such lower bound in \eqref{eq:dvtlower} to get
$$
d_{TV}(\PP^1_{\theta^0_1,g},\PP^1_{\theta^0_1,g^0}) \geq c 
 \sum_{k \in \Z} |c_k(g-\tilde{g})|^2  e^{-c k \log k}.
$$
for $c$ sufficiently small. 
We now  end the proof of the Theorem: choose a frequency cut-off $k_n$ that depends on $n$ and remark that
\begin{align*}
\forall g \in \Mnu \quad 
\|g-g^0\|^2 &= \sum_{|\ell| \leq k_n} |c_{\ell}(g-g^0)|^2 + \sum_{|\ell|>k_n} |c_{\ell}(g-g^0)|^2  \\
& \lesssim e^{c k_n \log k_n} \sum_{|\ell| \leq k_n} |c_{\ell}(g-g^0)|^2 e^{-c \ell \log \ell} +  k_n^{-2\nu} \\ 
& \lesssim e^{c k_n \log k_n}d_{TV}(\PP^1_{\theta^0_1,g},\PP^1_{\theta^0_1,g^0}) + k_n^{-2\nu}.
\end{align*}
We know from Equation \eqref{eq:tronc2} that the last bound is smaller than 
$e^{c k_n \log k_n} \epsilon_n^{1/3} + k_n^{-2\nu}$ up to a multiplicative constant, with probability close to $1$ as $n$ goes to $+ \infty$. The optimal choice for $k_n$ yields
$$
[k_n +2\nu] \log k_n = \frac{1}{3} \log \frac{1}{\epsilon_n}.
$$
This thus ensures that
$$
\Pi_n \left\{g 
\, s.t.\,  \|g-g^0\|^2 \leq  M \log (n) ^{-2 \nu} \vert Y_1, \ldots Y_n \right\} \longrightarrow 1 \, \text{as}  \,  n \rightarrow + \infty. 
$$
\end{proof}
In the last proof, we have used the knowledge of $\nu$ as well as the radius $A$ of the Sobolev space $\Mnu$ in the last lines to build a suitable threshold $k_n$.
Without this assumption, we cannot control easily from the behaviour of the posterior distribution around $\PP_{f^0,g^0}$ the posterior weights on $\Mnu$: that's why it is difficult to conclude with an adaptive prior.

\subsubsection[Posterior contraction rate around f0]{Posterior contraction rate around $f^0$}

We then aim to obtain a similar result for the posterior weight on neighbourhoods of $f^0$. Even if our results are quite good for the first coefficient $\theta_1$, we will see that indeed,  this is far from being the case for the rest of its Fourier expansion. 

\begin{theo}\label{theo:contractionf}
 Assume $(f^0,g^0) \in \cF_s \times \Mnu$ and 
 \[\exists (c) > 0 \quad \exists \beta > \nu+\tfrac{1}{2} \quad \forall k \in \mathbb{Z} \qquad |\theta_k(g^0)| \geq c k^{-\beta}, \]
 then
$$
\Pi_n \left( f  : \|f-f^0\|^2 > M \left( \log n \right)^{-2s \times \frac{2\nu}{2s+2\nu+1}}  \vert Y_1, \ldots, Y_n \right) \longrightarrow 0
$$
in $\PP_{f^0,g^0}$ probability as $n \rightarrow + \infty$, for a sufficiently large $M$.
\end{theo}

\begin{proof}
The idea of the proof is very similar to the former used arguments, we aim to study the posterior weight on neighbourhoods of the true Fourier coefficients of $f^0$, whose frequency is larger than $1$.

\paragraph{Point 1: Triangular inequality}
For any $f \in \cF_s$, we have for any $k \in \Z$:
$$d_{TV}(\PP^k_{f,g^0},\PP^k_{f^0,g^0})  \leq d_{TV}(\PP^k_{f,g^0},\PP^k_{f,g})  + d_{TV}(\PP^k_{f^0,g^0},\PP^k_{f,g}).
$$
The second term does not exceed $\epsilon_n \ll\log(n) ^{-\nu} $ with a probability tending to $1$, more precisely
\begin{equation}\label{eq:bound_dvt_triangle1}
\Pi_n \left( \left. 
\forall k \in \mathbb{Z} \quad  d_{TV}(\PP^k_{f,g},\PP^k_{f^0,g^0}) < M \epsilon_n \right\vert Y_1, \ldots, Y_n\right)  \longrightarrow 1
\end{equation}
as $n \longrightarrow +\infty$.

\paragraph{Point 2: $\displaystyle
\Pi_n \left( \left. 
\sup_{ k \in \mathbb{Z}} d_{TV}(\PP^k_{f,g^0},\PP^k_{f,g}) < M \log(n)^{- \nu} \right\vert Y_1, \ldots, Y_n \right) 
\rightarrow
1$}

To obtain such a limit, we can  use first the Cauchy-Schwarz inequality as follows:
\begin{align*}
d_{TV}(\PP^k_{f,g^0},\PP^k_{f,g}) & = \frac{1}{2\pi} \int_{\mathbb{C}} \left|\int_{0}^{2 \pi}
e^{-|z-\theta_k e^{\i k \varphi}|^2} [g(\varphi) - g^0(\varphi)] d \varphi \right| dz\\
& \leq \frac{\|g-g^0\|}{2 \pi} \int_{\mathbb{C}} \left[\int_{0}^{2 \pi}
e^{-2 |z-\theta_k e^{\i k \varphi}|^2} d \varphi \right] ^{1/2}dz
\end{align*}
Now, the Young inequality implies that for any $M>0$,
$$
|z-\theta_k e^{\i k \varphi}|^2 = |z|^2 + |\theta_k|^2 - 2 \Re \left( \bar{z} \theta_k e^{\i k \varphi}\right) \geq |z|^2 \left(1-\frac{1}{M}\right)- |\theta_k|^2(M-1)
$$
and the choice $M=2$ yields
\begin{equation}\label{eq:bound_dvt_triangle3}
d_{TV}(\PP^k_{f,g^0},\PP^k_{f,g})  \leq \frac{\|g-g^0\|}{2\pi} \int_{\mathbb{C}}
 \left( e^{-|z|^2 + 2|\theta_k|^2} \right)^{1/2} dz \leq 
 \|g-g^0\| \frac{e^{|\theta_k|^2}}{2}.
\end{equation}

To obtain that the former term is bounded, we first establish that indeed the posterior distribution asymptotically only weights functions $f$ with bounded Fourier coefficients.
We hence denote $$\mathcal{A}_n = \{(f,g) : \exists k \in \mathbb{Z} \quad  d_{TV}(\PP^k_{f,g^0},\PP^k_{f,g})  \geq M \log (n)^{-\nu}  \}$$ and the two sets
$$\mathcal{B} = \{f  : \forall k \in \mathbb{Z} \quad |\theta_k| \leq |\theta_k^0|+M \log(n)^{-\nu}\}$$ and 
$$\mathcal{C} = \{f  : \forall k \in \mathbb{Z} \quad |\theta_k^0| \leq |\theta_k|+M \log(n)^{-\nu}\}.$$
We first consider an integer $k$ and $\theta_k$ such that $|\theta_k|>|\theta_k^0|+M\log(n)^{-\nu}$, then
$$
d_{TV}(\PP^k_{f^0,g^0},\PP^k_{f,g})  = \frac{1}{2\pi} \int_{\mathbb{C}} \left|\int_{0}^{2\pi}\left[
e^{-|z-\theta_k e^{\i k \varphi}|^2} g(\varphi) - e^{-|z-\theta^0_k e^{\i k \varphi}|^2} g^0(\varphi)\right] d \varphi \right| dz.
$$
For any $z$ in the centered complex ball $B_{n}=B\left(0,\frac{M \log(n)^{-\nu}}{3}\right)$, one has for any $\varphi \in [0,2\pi]$
\begin{align*}
 |z-\theta^0_k e^{\i k \varphi}| \leq   \frac{M \log(n)^{-\nu}}{3}+|\theta_k^0|& \leq 2 \frac{M \log(n)^{-\nu}}{3}+|\theta_k^0|\\
 &  \leq |\theta_k| -\frac{M \log(n)^{-\nu}}{3}\leq |z-\theta_k e^{\i k \varphi}|.
\end{align*}
Hence if $|\theta_k| \geq |\theta_k^0|+M\log(n)^{-\nu}$, one has
\begin{multline*}
 d_{TV}(\PP^k_{f^0,g^0},\PP^k_{f,g}) \\
 \begin{aligned}
  &\geq \frac{1}{2\pi} \int_{B_{n}} \left|\int_{0}^{2\pi} \left[e^{-|z-\theta_k e^{\i k \varphi}|^2} g(\varphi) - e^{-|z-\theta^0_k e^{\i k \varphi}|^2} g^0(\varphi) \right]d \varphi \right| dz\\
  & \geq \frac{1}{2\pi} \int_{B_{n}} e^{-[|\theta_k^0|+M\log(n)^{-\nu}/3]^2} - e^{-[|\theta_k^0|+2M\log(n)^{-\nu}/3]^2} dz \\
  & \geq c |\theta_k^0|^2 e^{-|\theta_k^0|^2} \log(n)^{-3\nu},
 \end{aligned}
\end{multline*}
for a sufficiently small absolute constant $c>0$. Since the sequence $(\theta_k^0)_{k \in \mathbb{Z}}$ is bounded, for $n$ large enough, we know that 
$\|\theta^0\|^2 e^{-\inf_{k}|\theta_k^0|^2} \log(n)^{-3\nu} \gg \epsilon_n$. We can deduce from \eqref{eq:bound_dvt_triangle1} that

\begin{equation}\label{eq:bound_dvt_triangle2}
\Pi_n \left(\mathcal{B}^c  \vert Y_1, \ldots, Y_n\right)
   \longrightarrow 0 \quad  \text{as} \quad n \longrightarrow +\infty.
\end{equation}
A similar argument yields
$$\Pi_n \left(\mathcal{C}^c  \vert Y_1, \ldots, Y_n\right)
   \longrightarrow 0 \quad  \text{as} \quad n \longrightarrow +\infty.
$$
Gathering now \eqref{eq:bound_dvt_triangle2} and \eqref{eq:bound_dvt_triangle3}, 
we get for a sufficiently large $M$ 
\begin{align*}
\Pi_n \left( \mathcal{A}_n    \vert Y_1, \ldots, Y_n\right) &=
\Pi_n \left( \mathcal{A}_n  \cap \mathcal{B} \cap \mathcal{C} \vert Y_1, \ldots, Y_n\right)  \\
&\quad 
+ \Pi_n \left( \mathcal{A}_n  \cap (\mathcal{B} \cap \mathcal{C})^c \vert Y_1, \ldots, Y_n\right)  \\
& \leq \Pi_n \left( \|g-g^0\| \geq M e^{-\left(1+\sup_{k} |\theta_k^0|^2\right)}\log (n)^{-\nu} \right) \\
&\quad
+  \Pi_n \left(  \mathcal{B}^c \vert Y_1, \ldots, Y_n\right)  +  \Pi_n \left(  \mathcal{C}^c \vert Y_1, \ldots, Y_n\right)  
\end{align*}
We can now apply Theorem \ref{theo:contraction_g} to obtain the desired result:
\begin{equation}\label{eq:dvt_impt}
\Pi_n \left(\left. \sup_{k\in\Z} d_{TV}(\PP^k_{f,g^0},\PP^k_{f,g}) < M \log(n)^{- \nu} \right\vert Y_1, \ldots, Y_n \right)  \longrightarrow 1
\end{equation}
as $n \longrightarrow +\infty$.

\paragraph{Point 3: Contraction of $\theta_k$ near $\theta_k^0$} 
From the arguments of Point 2, we see that
$$ \Pi_n \left( f: \forall k \in \mathbb{Z} \quad \left||\theta_k| - |\theta_k^0| \right| <  M \log(n)^{- \nu} \right)  \longrightarrow 1 \quad  \text{as} \quad n \longrightarrow +\infty.$$
We now study the situation when $\left||\theta_k| - |\theta_k^0| \right| < M \log(n)^{-\nu}$, and we can write $\theta_k=\theta_k^0 e^{\i \varphi} + \xi_n$ where $\xi_n$ is a complex number such that $|\xi_n| \leq M \log(n)^{-\nu}$.
$$   
d_{TV}(\PP^k_{f,g^0},\PP^k_{f^0,g^0})=\frac{1}{2\pi} \int_{\C} \underbraceabs{\int_{0}^{2\pi} \left[ e^{-| z - \theta_k e^{i k \alpha} |^2} - e^{- | z - \theta_k^0 e^{i k \alpha}|^2}\right]  g^0(\alpha) d\alpha }{:=F(z)} dz$$
Indeed, $F(0) \simeq0$ since a Taylor expansion near $0$ yields at first order in $z$ and $\xi_n$ that
\begin{align*}
F(z)& = 2 e^{-|\theta_k^0|^2} \int_{0}^{2 \pi}\left[1+ \re \left(z \bar{\theta}_k e^{- i k \alpha} \right) -(1+ \re \left(z \bar{\theta}^0_k e^{- i k \alpha} \right) \right] g^0(\alpha)d \alpha\\ &\quad + o(|z|)+\cO(|\xi_n|).
\end{align*}
If one uses now $\theta_k = \theta_k^0 e^{\i \varphi} + \mathcal{O}(\log(n)^{-\nu})$, the computation of the integral above yields for $c<2$ and $\eta$ small enough such that $|z|\leq \eta$:
$$
|F(z)| \geq  c e^{-|\theta_k^0|^2}  \left|\sin (\varphi/2) \re \left(z \i e^{\i \varphi/2} \bar{\theta}_k^0 c_{-k}(g^0) \right) \right| + \mathcal{O}(\log(n)^{-\nu})
$$
Now, denote $\bar{u} = \frac{\i e^{\i \varphi/2} \bar{\theta}_k^0  c_{-k}(g^0)}{|\theta_k^0| \times |c_{-k}(g^0)|}$ which is a complex number of norm $1$, and let $v= \bar{u} e^{\i \pi/2}$. The vector $v$ is orthogonal to $\bar{u}$ and $z$ may be decomposed as
$$
z = a \bar{u} + b v.
$$
We then choose  $|b|<|a|/2$ and denotes $\mathcal{R}_a$ the area where $z$ is living. For $a < \eta$ small enough, we obtain that there exists an absolute constant $c$ independent of $k$ such that
\begin{align*}
d_{TV}(\PP^k_{f,g^0},\PP^k_{f^0,g^0}) &\geq  \int_{\mathcal{R}_a}|F(z)| \geq  c \eta^3 e^{- |\theta_k^0|^2} |\sin(\varphi/2)| |\bar{\theta}_k^0| | c_{-k}(g^0)|\\ &\quad + \mathcal{O}\left(\log(n)^{-\nu}\right).
\end{align*}
 Since 
$|\theta_k-\theta_k^0| =  2  |\sin(\varphi/2)| |\theta_k^0|+\mathcal{O}\left(\log(n)^{-\nu}\right)$, we  get that :
\begin{equation}\label{eq:mino_dvt2}
d_{TV}(\PP^k_{f,g^0},\PP^k_{f^0,g^0}) \geq c \eta^3 e^{- |\theta_k^0|^2} | c_{-k}(g^0)| |\theta_k-\theta_k^0|+\mathcal{O}\left(\log(n)^{-\nu}\right).
\end{equation}
Thus, we can conclude using \eqref{eq:dvt_impt} and \eqref{eq:mino_dvt2}  that there exists a sufficiently large $M$ such that
\begin{equation}\label{eq:contraction_thetak}
\Pi_n \left(\left. f : \sup_{k \in \mathbb{Z}} \left|(\theta_k-\theta_k^0) c_{-k}(g^0)\right|  < M \log(n)^{- \nu} \right\vert Y_1, \ldots, Y_n \right)  \longrightarrow 1 
\end{equation}
as $n \longrightarrow +\infty$.

\paragraph{Point 4: Contraction on $f^0$} 

We can now produce a  very similar proof to the one used at the end of Theorem \ref{theo:contraction_g}:
\begin{align*}
\|f-f^0\|^2&  =  \sum_{|\ell| > k_n} |\theta_{\ell}-\theta_{\ell}^0|^2 + \sum_{|\ell| \leq k_n} |\theta_{\ell}-\theta_{\ell}^0|^2\\
& \lesssim  k_n^{-2s} + \sum_{|\ell|\leq k_n} 
\frac{|\theta_{\ell}-\theta_{\ell}^0|^2|c_{-\ell}(g^0)|^2 }{ |c_{-\ell}(g^0)|^2}\\
& \lesssim  k_n^{-2s} +k_n^{2 \beta}
 \sum_{|\ell|\leq k_n} 
|\theta_{\ell}-\theta_{\ell}^0|^2|c_{-\ell}(g^0)|^2 \\
& \lesssim   k_n^{-2s} +k_n^{2 \beta+1} \sup_{|\ell| \leq  k_n } |\theta_{\ell}-\theta_{\ell}^0|^2|c_{-\ell}(g^0)|^2 
\end{align*}
Hence,  \eqref{eq:contraction_thetak} implies
$$
\Pi_n \left(\left. f : \|f-f^0\|^2 \leq k_n^{-2s} +k_n^{2 \beta+1} \log(n)^{-2\nu} \right\vert Y_1, \ldots, Y_n\right)  \longrightarrow 1 \, \text{as} \, n \longrightarrow +\infty.
$$
The optimal choice of the frequency cut-off  is $k_n = \left( \log n\right)^{\frac{2\nu}{2 \nu + 2s+1}}$, which yields
$$
\Pi_n \left(\left. f : \|f-f^0\|^2 \leq M \left(\log n \right)^{-4s\nu/(2s+2\beta+1)} \right\vert Y_1, \ldots, Y_n \right) \longrightarrow 1 
$$
as $n \longrightarrow +\infty$. This last result is the desired inequality.
\end{proof}

\begin{remark}
The lower bound obtained on $d_{TV}(\PP^k_{f,g^0},\PP^k_{f^0,g^0})$ will be important to understand how one should build an appropriate net of functions $(f_j,g_j) \in \cF_s  \times \Mnu$ hard to distinguish according to the $L^2$ distance. 
When $|\theta_k| \neq |\theta_k^0|$, it is quite easy to distinguish the two hypotheses but it is far from being the case when their modulus is equal. In such a case, the behaviour of the Fourier coefficients of $g^0$ becomes important. 
This is a clue to exhibit an efficient lower bound through the Fano lemma (for instance). This is detailed in the next paragraph. 
\end{remark}

\section{Lower bound from a frequentist point of view\label{sec:lowerbound}}

\subsection{Link with the convolution with unknown variance situation}
We complete now our study of the Shape Invariant Model by a small investigation on how one could obtain some lower bounds in the frequentist paradigm. We could consider several methods. Among them, the first one could be the use of results in the literature, such as the works of \cite{M02} or \cite{BM05}. Indeed, in the convolution model with unknown variance
\begin{equation}\label{eq:matias}
Y_i = X_i+\epsilon_i, \forall i \in \{1 \ldots n\} \qquad (X_i)_{i=1 \ldots n} \sim g,\end{equation}
 we already know that one cannot beat some $\log n$ power for the convergence rate of any estimator of both $g$ and of the variance of the noise $\sigma^2$. Such a nice result is obtained using the so-called  van Trees  inequality  which is a Bayesian Cramer-Rao bound (see for instance \cite{GL95} for further details). 
However their result cannot be used here:  Proposition \ref{prop:theta1} p. \pageref{prop:theta1} is much more optimistic since we obtain there a polynomial rate for the posterior contraction around $\theta_1^0$.
  
 First, note the results given by \cite{M02} and Proposition \ref{prop:theta1} are not opposite. Indeed, \cite{M02} considers lower bounds in a larger class than the estimation problem of $\theta_1$ written as \eqref{eq:deconvolution_unknown}: from a minimax point of view, the supremum over all hypotheses is taken in a somewhat larger set than ours. 
Moreover, if one considers \eqref{eq:deconvolution_unknown}, the density of $e^{- \i 2 \pi \tau_j}$ is supported by $\mathbb{S}^1$ instead of the whole complex plane which would be a natural extension of \eqref{eq:matias}. 
Hence, $g$ is a singular measure with respect to the noise measure: the ability of going beyond the logarithmic convergence rates is certainly due to the degeneracy nature of our problem according to the Gaussian complex noise. It is an important structural information which is not available when one considers general problems such as \eqref{eq:matias}.

\subsection{Lower bound}

Following such considerations, we are thus driven to build some nets of hypotheses hard to distinguish between and then apply some classical tools for lower bound results. 
We have chosen to use the Fano Lemma (see \cite{ibragimov_book} for instance) instead of Le Cam's method, since we will only be able to find some \textit{discrete} (instead of convex) set of pairs $(f_j,g_j)$ in $\cF_s \times \Mnu$ closed according to the Total Variation distance. We first recall the version of the Fano Lemma we used.
\begin{lemma}[Fano's Lemma]\label{lemma:fano} Let $r \geq 2$ be an integer and $\cM_r  \subset \cP $ which contains $r$ probability distributions indexed by $j=1 \ldots r$ such that
$$
\forall j \neq j' \qquad d(\theta(P_j),\theta(P_{j'}) \geq \alpha_r,
$$
and 
$$
d_{KL}(P_j,P_{j'}) \leq \beta_r.
$$
Then, for any estimator $\hat{\theta}$, the following lower bound holds
$$
\max_{j} \EE_{j} \left[d(\hat{\theta},\theta(P_j)) \right] \geq \frac{\alpha_r}{2} \left(1 - \frac{\beta_r + \log 2}{\log r} \right).
$$
\end{lemma}

We derive now our lower bounds result.

\begin{theo} \label{theo:lowerbound}
There exists a sufficiently small $c$ such that the minimax rates of estimation over $ \cF_s \times \Mnu$ satisfy
$$
\liminf_{n \rightarrow + \infty} \left( \log n\right)^{2s+2}
\inf_{\hat{f} \in \cF_s}  \sup_{(f,g) \in \cF_s \times \Mnu} \|\hat{f} - f\|^2 \geq c,
$$
and
$$
\liminf_{n \rightarrow + \infty} \left( \log n\right)^{2\nu + 1}
\inf_{\hat{g} \in \cF_s}  \sup_{(f,g) \in \cF_s \times \Mnu} \|\hat{g} - g\|^2 \geq c .
$$
\end{theo}

\begin{proof}
We will adapt the Fano Lemma to our setting and we are looking for a set $(f_j,g_j)_{j=1 \ldots p_n}$ such that each $\PP_{f_j,g_j}$ are closed together with rather different functional parameters $f_j$ or $g_j$. Reading carefully the Bayesian contraction rate is informative to build $p_n$ hypotheses which are difficult to distinguish. First, we know that since each $f_j$ should belong to $\cF_s$, we must impose for any $f_j$  that $\theta_1(f_j)>0$. 
From Proposition \ref{prop:theta1}, we know that one can easily distinguish two laws $\PP_{f_j,g_j}$ and $\PP_{f_{j'},g_{j'}}$ as soon as $\theta_1(f_j) \neq \theta_1(f_{j'})$. Then we build our net using  a common choice for the first Fourier coefficient of each $f_j$ in our net. For instance, we impose that
$$
\forall j \in \{1 \ldots p_n\} \qquad \theta_1(f_j) = 1.
$$

\paragraph{Point 1: Net of functions $(f_j)_{j = 1\ldots p_n}$}
We choose the following construction
\begin{equation}\label{eq:def_net}
\forall j \in \{1 \ldots p_n\}  \qquad f_j (x) = e^{\i 2\pi x} + p^{-s}_n e^{\i 2 \frac{(j-1)}{p_n} \pi} e^{\i 2 \pi p_n x}.
\end{equation}
The number of elements in the net $p_n$ will be adjusted in the sequel and will grow to $+ \infty$. Note that our construction naturally satisfies that each net $(f_j)_{j =1 \ldots p_n}$ belongs to $\cF_s$ since the modulus of the $p_n$-th Fourier coefficient is of size $p_n^{-s}$. At last, we have the following rather trivial inequality: $
\forall (j,j') \in \{1 \ldots p_n\}^2$ 
$$  \|f_j - f_{j'}\|^2 \geq p_n^{-2s} \times \left| e^{\i 2 \pi/p_n} - 1\right|^{2} \geq 4 p_n^{-2s} \sin^2 ( \pi/p_n) \sim_{n \mapsto +\infty} 4\pi^2 p_n^{-2s-2}.
$$

\paragraph{Point 2: Net of measures $(g_j)_{j = 1\ldots p_n}$}
The core of the lower bound is how to adjust the measures of the random shifts to make the distributions $\PP_{f_j,g_j}$, $j = 1 \ldots p_n$, as close as possible. First, remark that the Fano Lemma \ref{lemma:fano} is formulated with entropy between laws although it is quite difficult to handle when dealing with mixtures. In the sequel, we will choose to still use the Total Variation distance, and then use the chain of inequalities: $\forall j \neq j'$
\begin{eqnarray*}
d_{TV} \left( \PP_{f_j,g_j} , \PP_{f_{j'},g_{j'}}\right)  \leq \eta & \Rightarrow& d_{H} \left( \PP_{f_j,g_j} , \PP_{f_{j'},g_{j'}}\right) \leq \sqrt{2\eta}\\ &\Rightarrow & d_{KL} \left( \PP_{f_j,g_j} , \PP_{f_{j'},g_{j'}}\right) \lesssim \sqrt{\eta} \log \frac{1}{\eta}.
\end{eqnarray*}
Hence, from the tensorisation of the entropy, we must find a net such that $d_{TV} \left( \PP_{f_j,g_j} , \PP_{f_{j'},g_{j'}}\right)  \leq \eta_n$ with $-\sqrt{\eta_n} \log \eta_n =\cO( 1/n)$ to obtain a tractable application of the Fano Lemma (in which $P_j = \PP_{f_j,g_j}^{\otimes n}$). It imposes to find some mixture laws such that 
$d_{TV} \left( \PP_{f_j,g_j} , \PP_{f_{j'},g_{j'}}\right) \lesssim \frac{1}{(n \log n)^2}$. From the triangular inequality, it is sufficient to build $(g_j)_{j = 1 \ldots p_n}$ satisfying
\begin{equation}\label{eq:condTV}
\forall j \in \{1 \ldots p_n\} \qquad 
d_{TV} \left( \PP_{f_j,g_j} , \PP_{f_1,g_1}\right) \lesssim \frac{1}{(n \log n)^2}.
\end{equation}
For sake of convenience, we will omit the dependence of $p_n$ on $n$ and simplify the notation to $p$. In a similar way, $\theta_p^j $ will denote the $p$-th Fourier coefficient of $f_j$ given by
$\theta_p^j = e^{\i 2 \pi \alpha_j } \theta_p^1$ where $\alpha_j=\frac{j-1}{p_n}$.
From our choice of $(f_j)_{j = 1 \ldots p_n}$ given by \eqref{eq:def_net}, we have
\begin{align*}
d_{TV} \left( \PP_{f_j,g_j} , \PP_{f_1,g_1}\right) &= \frac{1}{2\pi^2} \int_{\C \times \C} \left| 
\int_{0}^{1}
e^{-|z_1 - e^{\i 2 \pi \varphi}|^2 - |z_2 - e^{\i 2 \pi p \varphi}\theta_p^1|^2}  g_1(\varphi)   d \varphi
\right. \\ 
& \left.
 - \int_{0}^{1} e^{-|z_1 - e^{\i 2 \pi \varphi}|^2 - |z_2 - e^{\i 2 \pi p \varphi}\theta_p^j|^2}  g_j(\varphi) d \varphi \right| dz_1 dz_2
\end{align*}

We will use the smoothness of Gaussian densities to obtain a suitable upper bound. 
Call $F$ the function defined on $\R^4$ by
$$
 F(x_1,y_1,x_2,y_2): = \int_{0}^{1} \left( e^{-\|z-\theta^1 \bullet \varphi\|^2} g_1(\varphi) - e^{-\|z-\theta^j \bullet \varphi\|^2}  g_j(\varphi) \right) d \varphi,
$$
where $z=(x_1+\i y_1,x_2+\i y_2)$ and $\theta^j \bullet \varphi=(e^{\i 2\pi \varphi}, \theta_j^p e^{\i 2\pi p \varphi})$.

To control $F$, we adapt the proof of Proposition 3.3 of \cite{BG_1}. Only the sketch of the proof for this point is given here, please see \cite{BG_1} for the details. We use a truncature for $(x_1,x_2,y_1,y_2) \in\mathcal{R}_{R_n} :=  B_{\R^2}(0,R_n)^2$. Outside $\mathcal{R}_{R_n}$, we use the key inequality (that comes from a Taylor expansion):
\begin{equation}\label{eq:expo}
\forall k \in \N \quad 
\forall y \in \R_+ \qquad\underbraceabs{e^{-y} - \sum_{j=0}^{k-1} \frac{(-y)^j}{j!}}{:=R_k(y)}
\leq \frac{|y|^k}{k!} \leq \frac{(e |y|)^k}{k^k}.
\end{equation}
Inside $\mathcal{R}_{R_n}$ we need to satisfy some constraints on the Fourier coefficients. 
Since here the only non null Fourier coefficients are of order $1$ and $p$, we have finally to ensure that
\begin{equation} \label{eq:liens_coefs_gj}
 \forall m, l \leq d \quad  \forall (s,\tilde{s}) \in \{-1; +1\}^2 \qquad c_{s m+\tilde{s} \ell p}(g_j) e^{\tilde{s} \ell  \alpha_j} = 
c_{s m+\tilde{s} \ell p}(g_1).
\end{equation}
Hence, the maximum size of $d$ is $d=p/4$. We have
\begin{align*}
 d_{TV} \left( \PP_{f_j,g_j} , \PP_{f_1,g_1}\right) &= \frac{1}{2\pi^2}  \int_{\mathcal{R}_{R_n}} |F(x_1,y_1,x_2,y_2) | dx_1  dy_1  dx_2  dy_2   \\ & \quad 
 + \frac{1}{2\pi^2} \int_{\mathcal{R}^c_{R_n}} |F(x_1,y_1,x_2,y_2) | dx_1  dy_1  dx_2  dy_2 \\
  & \lesssim e^{- R_n^2/2} + \left(\frac{(e R_n)^{p/4}}{(p/4)^{p/4}}\right)^{4} \lesssim e^{- R_n^2/2}  + \frac{(e R_n)^p}{(p/4)^p},
\end{align*}
where the last point is deduced from inequality \eqref{eq:expo}.
We choose now $R_n$ such as $R_n := 3 \sqrt{ \log n}$ to obtain that $e^{-R_n^2/2} \ll (n \log n)^{-2}$ as required in condition \eqref{eq:condTV}. Now, we control the last term of the last inequality: the Stirling formula yields
$$
\frac{(e R_n)^p}{(p/4)^p} \lesssim e^{p \log (3 \sqrt{\log n}) - p \log p/4}. $$
 If one chooses $p_n = \kappa \log n$ with $\kappa>12$, we then obtain that
$$ d_{TV} \left( \PP_{f_j,g_j} , \PP_{f_1,g_1}\right) \lesssim  e^{-C p_n \log p_n} \lesssim (n \log n)^{-2}.$$
Such a choice of $R_n$ and $p_n$ ensures that
\eqref{eq:condTV} is fulfilled.

We have to make sure that our conditions \eqref{eq:liens_coefs_gj} for the Fourier coefficients of the $g_j$'s lead to valid densities. Take for instance, for some $\beta > \nu + 1/2$,
\[ a = \frac{A}{\left(2\sum_{k\geq 1} k^{-2\beta+2\nu}\right)^{1/2}} \wedge \frac{1}{\left(2\sum_{k\geq 1} k^{-2\beta}\right)^{1/2}}. \]
Then take $c_0(g_j) :=1$, and $\forall k \in \Z^{\star}$, $c_k(g_1) := a |k|^{-\beta}$. This ensures that
\[ \sum_{k \in \Z^{\star}} |c_k(g_j)| \leq 1, \] and therefore all $g_j$ remains nonnegative. \\
Note that the densities $g_j$ fulfill the condition appearing in Theorem \ref{theo:semiparametric_rates}; the lower bounds below are also valid in this slightly smaller model.

We then conclude our proof: we aim to apply the Fano Lemma (see Lemma \ref{lemma:fano}) with $\alpha_n=p_n^{-2s-2}$ and $\beta_n = 
\cO(1)$ for the parametrization of $(f_j)_{j=1 \ldots p_n}$. We then deduce the first lower bound
$$
\liminf_{n \rightarrow + \infty} \left( \log n\right)^{2s+2}
\inf_{\hat{f} \in \cF_s}  \sup_{(f,g) \in \cF_s \times \Mnu} \|\hat{f} - f\|^2 \geq c.
$$

Our construction implies also that each $g_j$ are rather different each others since one has for instance,
$c_p(g_j) e^{\i \alpha_j} = c_p(g_1) = c_p(g_{j'}) e^{\i \alpha_{j'}}$. Thus
$$\forall j \neq j'  \qquad \|g_j-g_{j'}\|^2_2 \geq |c_p(g_j) - c_p(g_{j'})|^2 = p^{-2\nu} \left| e^{\alpha_j } - e^{\alpha_{j'} }\right|^2  \geq c p^{-2\nu-2}.
$$
Applying the Fano Lemma to $(g_j)_{j=1\ldots p_n}$ we get
$$
\liminf_{n \rightarrow + \infty} \left( \log n\right)^{2\nu+2}
\inf_{\hat{g} \in \cF_s}  \sup_{(f,g) \in \cF_s \times \Mnu} \|\hat{g} - g\|^2 \geq c.
$$
This ends the proof of the lower bound. 
\end{proof}

\section{Concluding remarks}

In this paper, we exhibit a suitable prior which enables to obtain a contraction rate of the posterior distribution near the true underlying 
distribution $\PP_{f^0, g^0}$.

Up to non restrictive condition, we can also obtain a large identifiability class 
to retrieve $f^0$ and $g^0$. 
However in this class the contraction of the posterior is dramatically damaged since we then obtain a logarithm rate instead of a polynomial one. This last point cannot be so much improved using the standard $L^2$ distance to measure the neighbourhoods of $f^0$ as pointed by our frequentist lower bounds. Remark that we do not obtain exactly the same rates for our lower and upper bounds of reconstruction. This may be due to the rough inequality  
$|    \psi_{a}(\varphi) |\geq \frac{|\psi_{a}(\varphi) |^2}{\|\psi_{a}\|_{\infty}}$ used to obtain
\eqref{eq:dvt_g_psi}. 

Indeed, the degradation of the contraction rate occurs when one tries to invert the identifiability map $\mathcal{I}: (f,g) \mapsto \mathbb{P}_{f,g}$. This difficulty should be understood as a novel consequence of the impossibility to exactly recover the random shifts parameters when only $n$ grows to $+\infty$. Such a phenomenon is highlighted in several papers such as \cite{BG10} or \cite{BGKM12}.

However, it may be possible to obtain a polynomial rate using a more appropriate distance adapted to our problem of randomly shifted curves 
$$
d_{Frechet}(f_1,f_2) := \inf_{\tau \in [0,1]} \|f_1^{-\tau}-f_2\|.
$$ 
We plan to tackle this problem in a future work. The important requirement in this view is to find some relations between the neighbourhoods of $\mathbb{P}_{f^0,g^0}$ and the neighbourhoods of $f^0$ according to the distance $d_{Frechet}$.


\appendix

\section{Small ball probability for integrated Brownian bridge}\label{sec:lower_bound}

In the sequel, we still use the notation $p_v$ defined by \eqref{eq:log_gaussian} to refer to the probability distribution which is proportionnal to $e^{v}$.
We detail here how one can obtain a lower bound of the prior weight around any element $g^0$. Since we deal with a log density model, it will be enough to find a lower bound of the weight around $w^0$ if one writes $g^0 \propto e^{w^0}$ according to Lemma \ref{prop:log_density} (which is the Lemma 3.1 of \cite{vdWvZ}).

\begin{lemma}[\cite{vdWvZ}]\label{prop:log_density}
For any real and measurable functions $v$ and $w$ of $[0,1]$, the Hellinger distance between $p_v$ and $p_w$ is bounded by
$$
d_H(p_v,p_w) \leq \|v-w\|_{\infty} e^{\|v-w\|_{\infty}/2}.
$$
\end{lemma}
We now obtain a lower bound of the prior weight on the set $\mathcal{G}_{\epsilon}$ previously defined as:
$$
\mathcal{G}_{\epsilon} := \left\{ g \in \mathfrak{M}_{\nu}([0,1])(2A) : d_{TV}(g,g^0) \leq \epsilon \right\}.
$$
This bound is given by the following Theorem.
\begin{theo}\label{theo:lower_bound_proba}
The prior $q_{\nu,A}$ defined by \eqref{eq:prior_process} and \eqref{eq:log_gaussian} satisfies for $\epsilon$ small enough:
$$
q_{\nu,A} \left( \mathcal{G}_{\epsilon} \right) \geq c e^{- \epsilon^{-\frac{1}{k_\nu+1/2}}},
$$
where $c$ is a constant which does not depend on $\epsilon$.
\end{theo}
\begin{proof}
The proof is divided in 4 steps.
\paragraph{Structure  of the prior}
We denote $w_0 := \log g^0$, which is a $k_\nu$-differentiable function of $[0,1]$, that can be extended to a $1$-periodic element of $\mathcal{C}^{k_\nu}(\mathbb{R})$.  We define $\tilde{q}$ the prior defined by \eqref{eq:prior_process} on such a class of periodic functions (and omit the dependence on $\nu$ and $A$ for sake of simplicity).
The prior $q_{\nu,A}$ is then derived from $\tilde{q}$ through \eqref{eq:log_gaussian}.
We can remark that our situation looks similar to the one described in paragraph 4.1 of \cite{vdWvZ} for integrated brownian motion. Indeed, the log-density $w_0$ should be approximated by some "Brownian bridge started at random" using
$$
w=J_{k_\nu}(B) + \sum_{i=1}^{k_{\nu}} Z_i \psi_i,
$$
where $B$ is a real Brownian bridge between $0$ and $1$. We suppose $B$ built as $B_t = W_t - t W_1$ on the basis of a Brownian motion $W$ on $[0, 1]$. 
 Of course, in the above equation, one can immediately check that $J_{k_\nu}(B)(0)=J_{k_\nu}(B)(1)=0$. Moreover, the relation 
 $J_{k}(f)'
 = J_{k-1}(f)
 - \int_{0}^1 J_{k-1}(f)
 $ and an induction argument yields
$$
\forall j \in \left\{1,\ldots, k_{\nu}\right\} \qquad J_{k_{\nu}}(B)^{(j)}(0)=J_{k_{\nu}}(B)^{(j)}(1).$$
Hence, $J_{k_\nu}(B)$ and its first $k_\nu$ derivatives are $1$-periodic.
Of course, the functions $\psi_i$ are also $1$-periodic and $\mathcal{C}^{\infty}(\mathbb{R})$ and thus our prior $\tilde{q}$ generates admissible functions of $[0,1]$ to approximate $w_0$. We will denote this set of admissible trajectories $\mathcal{C}^{k_{\nu}}_1$ to refer to $1$-periodic functions which are $k_{\nu}$ times differentiable.
 
 \paragraph{Transformation of the Brownian bridge}
 We denote $\mathbb{B}_1$ the separable Banach space of Brownian bridge trajectories between $0$ and $1$ and $\mathbb{B}_2 = \mathbb{R}^{k_{\nu}+1}$.
It is possible to check that the map
 $$
 T: (B,Z_0,\ldots,Z_{k_{\nu}}) \longmapsto J_{k_{\nu}}(B)+\sum_{i=0}^{k_{\nu}} Z_i  \psi_i $$
 is \textit{injective} from the Banach space $\mathbb{B}=\mathbb{B}_1 \times \mathbb{B}_2$ to the set $\underline{\mathbb{B}}:=T(\mathbb{B})$.
  More precisely, an recursive argument shows that each map $J_k(B)$ may be decomposed as
 \begin{equation} \label{eq:JkB}
  \forall t \in [0,1] \qquad J_{k}(B)(t)=I_k(W)(t) + \sum_{i=1}^{k+1} c_{i,k}(W) t^i,
 \end{equation}
 where $ c_{i,k}(W)$ are explicit linear functionals that depend on $W_1$ and on the collection 
 $\big(\int_0^1 (1-t)^{k-j} W_t d t\big)_{1\leq j\leq k}$ 
 (and not on $t$), and $I_k$ is the operator used in \cite{vdWvZ} defined as $I_1(f)=\int_{0}^t f$ and $I_k = I_{1} \circ I_{k-1}$ for $k \geq 2$. Hence,
 \begin{multline}\label{eq:jk_expand}
 \forall t \in [0,1], \quad 
 T(B,Z_1,\ldots,Z_{k_{\nu}})^{(k)}(t) \\
 = W_t +  c_{k,k}(W)k! +  c_{k+1,k}(W) (k+1)! t + \sum_{i=0}^{k_{\nu}} Z_i \psi_i^{(k)}(t)
 \end{multline}
According to the Brownian bridge representation \textit{via} its Karhunen-Loeve expansion (as sinus series), and since each $\psi_{i}^{(k)}$ possesses a non vanishing cosinus term: $t \mapsto \cos(2 \pi i t)$, we then deduce that
$$
 T(B^1,Z^1_1,\ldots,Z^1_{k_{\nu}}) =  T(B^2,Z^2_1,\ldots,Z^2_{k_{\nu}})
$$
necessarily implies that $Z^1_i=Z^2_i$ for  $i \in \{0, \ldots, k_{\nu}\}$, and next that $W^1=W^2$ and $B^1=B^2$.

 Thus, it is possible to apply Lemma 7.1 of \cite{vdWvZ2} to deduce that the Reproducing Kernel Hilbert Space (shortened as RKHS in the sequel) associated to the Gaussian process \eqref{eq:prior_process} in $\underline{\mathbb{B}}$ is $\underline{\mathbb{H}} := T \mathbb{H}$ where $\mathbb{H}$ is the RKHS derived in the simplest space $\mathbb{B}=\mathbb{B}_1 \times \mathbb{B}_2$. Moreover, the map $T$ is an isometry from $\mathbb{H}$ to $\underline{\mathbb{H}}$ for the RKHS-norms. At last, since the sets $\mathbb{B}_1$ and $\mathbb{B}_2$ are independent, the RKHS $\mathbb{H}$ may be described as
$$
\mathbb{H} := \left\{ (f,z) \in AC([0,1]) \times \mathbb{R}^{k_{\nu}+1} : f(0)=f(1)=0, \int_{0}^1 f'^2< \infty \right\},
$$ 
where $AC([0,1])$ is the set of absolutely continuous functions on $[0,1]$, $\mathbb{H}$ is endowed with the following inner product:
 $$
 \langle (f_1,z^1),(f_2,z^2)\rangle_{\mathbb{H}} := \int_0^1 f'_1 f'_2 + \langle z^1,z^2 \rangle_{\mathbb{R}^{k_{\nu}+1}}.
 $$

\paragraph{Extremal derivatives} 
We  study the influence of the process $$b:=\sum_{i=0}^{k_{\nu}} Z_i \psi_i$$ and are looking for realizations of $(Z_i)_{i}$ that suitably matches arbitrarily values $w_0^{(j)}(0) = w_0^{(j)}(1)$. In this view, simple computations yield that for any integer $p$:
$$
\psi^{(2p)}_k(t) = (-1)^{p} (2\pi k)^{2p} \psi_k(t),
$$
and
$$
\psi^{(2p+1)}_k(t) = (-1)^{p} (2\pi k)^{2p+1}[- \sin (2 \pi k t) + \cos (2 \pi k t)].
$$
Hence, the matching of $w_0^{(j)}(0)$ by $b^{(j)}(0)$ is quantified by
$$w_0^{(j)}(0) - b^{(j)}(0)= w_0^{(j)}(0) - \sum_{k=0}^{k_{\nu}} (-1)^{\lfloor j/2  \rfloor}  (2 \pi k)^{j} Z_k.
$$
If one denotes  $\alpha_k := 2 \pi k$, the vector of derivatives as $d_0:=(w_0^{(j)}(0))_{j=0 \ldots k_{\nu}}$, $Z=(Z_0,\ldots, Z_{k_{\nu}+1})$ and the squared matrix of size $(k_{\nu}+1) \times (k_{\nu}+1)$:
$$
A_0 := \left( 
\begin{matrix}
 1  & 1 & \hdots &  1 \\ \alpha_1 & \alpha_2 & \hdots & \alpha_{k_\nu} \\
 -\alpha_1^2  & - \alpha_2^2 & \ldots & - \alpha_{k_\nu}^2 \\
 -\alpha_1^3  & - \alpha_2^3 & \ldots & - \alpha_{k_\nu}^3 \\ 
\alpha_1^{4} &  \alpha_2^4 & \ldots &  \alpha_{k_\nu}^4 \\ 
 \vdots & & & \\
\end{matrix}
\right),
$$
then we are looking for values of $Z$ such that $d_0 = A_0 Z$.
The matrix $A_0$ is invertible since it may be linked with the Vandermonde matrix. 

We can now establish that the support of the prior (adherence of $\underline{\mathbb{B}}$) is exactly $\mathcal{C}_1^{k_{\nu}}$. Indeed, the support of the transformed Brownian bridge $J_k(B)$ is included in the set of $1$-periodic functions $\mathcal{C}_1^{k_{\nu}}$ which possesses at the most $k+1$ constraints on the values of their $k_{\nu}+1$ first derivatives at the point $0$. These constraints are given by the  coefficients $(c_{i,k_\nu})_{i=0 \ldots k_{\nu}}$ in \eqref{eq:jk_expand}. From the invertibility of the matrix $A_0$, it is possible to match \textit{any} term $w_0^{(j)}(0), 0 \leq j \leq k_{\nu}$ with the additional process $b$ \citep[see][section 10]{vdWvZ2}.

\paragraph{Small ball probability estimates}
We now turn into the core of the proof of the Theorem.
Since the Total Variation distance is bounded from above by the Hellinger distance, an immediate application of Lemma \ref{prop:log_density} shows that it is sufficient to find a lower bound of the $\tilde{q}(\tilde{\mathcal{G}}_{\epsilon})$ where
$$
\tilde{\mathcal{G}}_{\epsilon}:= \left\{w \in \mathcal{C}_{1}^{k_\nu}([0,1]) : \|w-w_0\|_{\infty} \leq \epsilon \right\}.
$$ 
Following the argument of \cite{KWL} on \textit{shifted} Gaussian ball, we have
$$
\log \left( 
\tilde{q}\left(\tilde{\mathcal{G}}_{\epsilon}\right) \right) \geq - \inf_{h \in \underline{\mathbb{H}} : \|h-w_0\|_{\infty} \leq \epsilon} \|h\|_{\underline{\mathbb{H}}}^2 - \log \tilde{q}\left( \|J_k(B)+b\|_{\infty} \leq \epsilon \right).
$$
From the isometry $T$ from $\mathbb{H}$ to $\underline{\mathbb{H}}$, we can write that the approximation term 
$\inf_{h \in \underline{\mathbb{H}} : \|h-w_0\|_{\infty} \leq \epsilon} \|h\|_{\underline{\mathbb{H}}}^2$ is of the same order as the approximation term that we can derive in $\mathbb{H}$, and the arguments of Theorem 4.1 in \cite{vdWvZ} can be applied here to get
$$
\inf_{h \in \underline{\mathbb{H}} : \|h-w_0\|_{\infty} \leq \epsilon} \|h\|_{\underline{\mathbb{H}}}^2 \lesssim \epsilon^{-\frac{1}{k_{\nu}+1/2}}.
$$

It reminds to obtain a lower bound of the small ball probability of the \textit{centered} Gaussian ball. Note that $b$ and $J_{k_{\nu}}$ are independent Gaussian processes. We have somewhat trivially that
$
\log \left( \frac{1}{\epsilon} \right) \lesssim \log \mathbb{P} \left( \|b\|_{\infty} \leq \epsilon \right).
$
Thus, the main difficulty relies on the lower bound of
$$
\phi_0(\epsilon):=\log \mathbb{P} \left( \|J_k(B)\|_{\infty} \leq \epsilon \right).
$$
Going back to \eqref{eq:JkB}, we see that $J_k(B)$ can be decomposed into two nonindependent Gaussian processes: $I_k(W)$ and a polynomial $\sum_{i=1}^{k+1} c_{i,k}(W) t^i$ which is a linear functional of $W_1$ and of the collection 
$\big(\int_0^1 (1-t)^{k-j} W_t d t\big)_{1\leq j\leq k}$. 
Therefore 
\[\log \left( \frac{1}{\epsilon} \right) \lesssim \log \mathbb{P} \left( \left\|
J_k(B) - I_k(W)
\right\|_{\infty} \leq \epsilon \right). \]
Now, applying Theorems 3.4 and 3.7 of \cite{Li_Shao} 
yields
\[ \log \mathbb{P} \left(\|J_{k_\nu}(B)\|_{\infty} \leq \epsilon  \right) \sim \log \mathbb{P} \left(\|I_{k_\nu}(W)\|_{\infty} \leq \epsilon  \right) \geq -\epsilon^{-\frac{1}{k_{\nu}+1/2}}, \]
which is of the same order as the approximation term. Gathering now our lower bound on shifted Gaussian ball and the term above ends the proof of the Theorem.
\end{proof}

\section{Equivalents on Modified Bessel functions}\label{sec:bessel_appendix}

\begin{lemma}\label{lemma:equi_bessel}  For any $n \in \Z$ and $a>0$, define
$$
A_n(a) := \int_{0}^{2\pi}e^{a \cos(u)} \cos (nu) d u.
$$ 
Then, the following equivalent holds:
$$
\forall a \in [0,\sqrt{n}] \qquad A_n(a) \sim \frac{2 \pi}{n!} \left(\frac{a}{2}\right)^n 
\left( 1+\mathcal{O}\left(\frac{a}{n} \right)\right).
$$
\end{lemma}
\begin{proof}
This equivalent is related to the modified Bessel functions (see \cite{AS64} for classical equivalents on Bessel functions and \cite{LL10} for standard results on continuous time random walks). More precisely, $I_m(a)$ is defined as
$$
\forall m \in \N, \forall a >0 \qquad 
I_m(a):= \sum_{k\geq0}\frac{1}{k!(k+m)!} \left(\frac{a}{2}\right)^{2k+m},
$$
and we have (see for instance \cite{AS64}) $$
I_0(a) +2 \sum_{m=1}^{+\infty}I_m(a) \cos (m u) = e^{a \cos u}.
$$
Hence, we easily deduce that $A_n(a) = 2 \pi I_n(a)$.
For small $a$, it is possible to use standard results on modified Bessel functions. Equation (9.7.7) of \cite{AS64}, p. 378. yields
\begin{equation}\label{eq:equi2}
\forall a \in [0,\sqrt{n}] \qquad 
I_n(a) \sim \frac{1}{n!} \left(\frac{a}{2}\right)^n \left( 1+\mathcal{O}\left( \frac{a}{n}\right) \right) .
\end{equation}
\end{proof}

This integral is strongly related to the density of continuous time random walk if one remark that if $B_n(a)=e^{-a} A_n(a)/(2\pi)$, one has
$
B_n(0)=0, \forall n \neq 0$ and $B_0(0)=1$ and at last
$$
B'_n(a) = \frac{B'_n(a-1)+B'_n(a+1)}{2} - B_n(a).
$$
Hence, $B_n(a)$ is the probability of a continuous time random walk to be in place $n \in \Z$  at time $a$. In this way, we  get some asymptotic equivalents of $B_n(a)$ (and so of $A_n(a)$): from the Brownian approximation of the CTRW , we should suspect that for $a$ large enough
\begin{equation}\label{eq:equi_ctrw}
B_n(a) \sim \frac{1}{\sqrt{2 \pi a}} e^{-n^2/(2a)}, \forall a \gg n^2.
\end{equation}
 Moreover, from \cite{AS64}, we know that 
\begin{equation}\label{eq:equi1}
I_n(a) \sim \frac{e^{a}}{\sqrt{2 \pi a}}, \quad \text{as soon as} \quad a \geq 2n,
\end{equation}
and this equivalent is sharp when $a$ is large enough: from equation (9.7.1) p. 377 of \cite{AS64}, we know that
$$
\forall a \geq 4 n^2 \qquad 
I_n(a) \geq \frac{1}{2} \times \frac{e^{a}}{\sqrt{2 \pi a}}.
$$
We remark that \eqref{eq:equi1} yields the heuristic equivalent suspected in \eqref{eq:equi_ctrw}: $B_n(a) = e^{-a} I_n(a) \sim \frac{1}{\sqrt{2 \pi a}}$,  although \eqref{eq:equi2} provides a quite different information for smaller $a$. 
We do not have purchase more investigation on this asymptotic since we will see that indeed, \eqref{eq:equi2} is much more larger than \eqref{eq:equi1}. 

For $a \in [\sqrt{n},2n]$, we do not have found any satisfactory equivalent of modified Bessel functions. Formula of \cite{AS64} is still tractable but yields some different equivalent which is not  "uniform enough" since we need to integrate this equivalent for our bayesian analysis.
 This is not  so important since we can see for our range of application that the most important weight belongs to the smaller values of $a$. 
 \section*{Acknowledgements} S. G. is indebted to Patrick Cattiaux, and Laurent Miclo for stimulating discussions related to some technical parts of this work. Authors also thank J\'er\'emie Bigot, Isma\"el Castillo, Xavier Gendre, Judith Rousseau and Alain Trouv\'e for enlightening exchanges.

\bibliographystyle{alpha}
\bibliography{paper}

\end{document}